\documentclass[final]{siamart250211}

\usepackage{mathtools}
\usepackage{cite}
\usepackage{amsfonts}
\usepackage{graphicx}
\usepackage{epstopdf}
\usepackage{algorithmic}
\usepackage{tikz}
\usepackage{pgfplots}
\pgfplotsset{compat=1.13}
\usepackage{xcolor}
\usepackage{caption}

\usetikzlibrary{decorations.pathreplacing,matrix,calc}

\usepackage[section]{placeins}

\definecolor{SPECorange}{rgb}{1.0,.5625,0}
\definecolor{SPECblue}{rgb}{0,0,0.75}
\definecolor{SPECred}{rgb}{0.75,0,0}
\definecolor{SPECgreen}{rgb}{0,0.75,0}
\definecolor{SPECblack}{rgb}{0.75,0.75,0.75}

\ifpdf
  \DeclareGraphicsExtensions{.eps,.pdf,.png,.jpg}
\else
  \DeclareGraphicsExtensions{.eps}
\fi

\AtBeginEnvironment{bmatrix}{\setlength{\arraycolsep}{2pt}}

\numberwithin{theorem}{section}

\newcommand{\TheTitle}{
A Fast and Stable Square-Root-Free Unitary Core-Chasing Algorithm
}
 
\newcommand{\TheAuthors}{M. Esquivel-Rosado, J. L. Aurentz, G. Barbarino, and M. C. Quintana}

\headers{\TheTitle}{\TheAuthors}

\title{{\TheTitle}
}

\author{
  M\'onica Esquivel-Rosado\thanks{Universidad de Huelva, Huelva, Spain
    (\email{monica.esquivel@dmat.uhu.es}).}
  \and
  Jared L. Aurentz\thanks{Universidad de Huelva, Huelva, Spain
    (\email{jared.aurentz@dci.uhu.es}).}
  \and
  Giovanni Barbarino\thanks{Universidad de Huelva, Huelva, Spain
    (\email{giovanni.barbarino@dci.uhu.es}).}
  \and
  Mar\'ia C. Quintana \thanks{Universidad de Huelva, Huelva, Spain
    (\email{mc.quintana@dci.uhu.es}).}
}

\newsiamthm{example}{Example}

\makeatletter
\newcommand{\tQ}{\mathpalette\tQ@{}}
\newcommand{\tQ@}[2]{%
  \ooalign{%
    \hfil$\m@th#1 Q$\hfil\cr
    \hidewidth\raisebox{-0.35ex}{$\m@th#1\widetilde{\phantom{Q}}$}\hidewidth\cr
  }%
}
\makeatother

\newcommand{\htQ}{\smash{\widehat{\tQ}}}

\newcommand{\hU}{\widehat{U}}
\newcommand{\hQ}{\widehat{Q}}

\newcommand{\hu}{\widehat{u}}
\newcommand{\hv}{\widehat{v}}
\newcommand{\vv}{vv}
\newcommand{\cc}{cc}
\newcommand{\sss}{ss}
\newcommand{\nn}{nn}
\newcommand{\hvv}{\widehat{\vv}}
\newcommand{\hcc}{\widehat{\cc}}
\newcommand{\hsss}{\widehat{\sss}}
\newcommand{\hc}{\widehat{c}}
\newcommand{\hs}{\widehat{s}}
\newcommand{\homega}{\widehat{\omega}}
\newcommand{\hsigma}{\widehat{\sigma}}
\newcommand{\hgamma}{\widehat{\gamma}}

\newcommand{\utQ}{\underline{\smash{\tQ}}}
\newcommand{\uQ}{\underline{\smash{Q}}}
\newcommand{\uU}{\underline{U}}
\newcommand{\uhtQ}{\underline{\smash{\htQ}}}
\newcommand{\uhQ}{\underline{\smash{\hQ}}}
\newcommand{\uhU}{\underline{\hU}}
\newcommand{\euv}{\epsilon_{uv}}
\newcommand{\ecs}{\epsilon_{cs}}
\newcommand{\esigma}{\epsilon_{\sigma}}
\newcommand{\egamma}{\epsilon_{\gamma}}
\newcommand{\erho}{\epsilon_{\rho}}

\newcommand{\rhoB}{\overline{\rho}}
\newcommand{\gammaB}{\overline{\gamma}}
\newcommand{\omegaB}{\overline{\omega}}
\newcommand{\sigmaB}{\overline{\sigma}}
\newcommand{\hgammaB}{\overline{\hgamma}}

\newcommand{\hsigmaB}{\overline{\hsigma}}
\newcommand{\cB}{\overline{c}}
\newcommand{\tB}{\overline{t}}

\newcommand{\uB}{\overline{u}}
\newcommand{\ux}{\underline{x}}
\newcommand{\ut}{\underline{t}}
\newcommand{\uz}{\underline{z}}
\newcommand{\un}{\underline{n}}
\newcommand{\uu}{\underline{u}}
\newcommand{\uv}{\underline{v}}
\newcommand{\uc}{\underline{c}}
\newcommand{\us}{\underline{s}}
\newcommand{\urho}{\underline{\smash{\rho}}}
\newcommand{\ugamma}{\underline{\smash{\gamma}}}

\newcommand{\usigma}{\underline{\sigma}}
\newcommand{\utB}{\overline{\ut}}
\newcommand{\huB}{\overline{\hu}}
\newcommand{\hcB}{\overline{\hc}}
\newcommand{\hH}{\widehat{H}}
\newcommand{\rerr}{\mu}
\newcommand{\uhu}{\underline{\hu}}
\newcommand{\uhv}{\underline{\hv}}
\newcommand{\uhc}{\underline{\hc}}
\newcommand{\uhs}{\underline{\hs}}
\newcommand{\uhgamma}{\underline{\smash{\hgamma}}}

\newcommand{\uhsigma}{\underline{\hsigma}}
\newcommand{\ope}{(1+O(\epsilon))}
\newcommand{\opr}{(1+O(\rerr))}
\newcommand{\re}{\operatorname{Re}}
\newcommand{\im}{\operatorname{Im}}
\newcommand{\fl}{\operatorname{fl}}

\ifpdf
\hypersetup{
  pdftitle={\TheTitle},
  pdfauthor={\TheAuthors}
}
\fi

\begin{document}

\maketitle

\begin{abstract}
Symmetric and unitary matrices are among the most important classes of
structured matrices admitting efficient and stable eigenvalue algorithms. Notable
examples include structure-preserving variants of the QR algorithm and their
square-root-free counterparts. For unitary matrices, these algorithms were
historically derived from recurrence relations for orthogonal polynomials on the
unit circle. A matrix-based derivation using core chasing was later given in
\cite{aurentz2015a}; this family of algorithms is backward stable, but the
underlying symmetries needed to cast out the square roots were not identified,
and the question of structured backward stability was left open. In this paper we
present a matrix-based derivation of the symmetric unitary core-chasing algorithm
and use it to obtain a square-root-free variant. We give a matrix-based proof of
backward stability and show that strict structured backward stability is
impossible in general. Open-source Fortran code and numerical experiments
demonstrate that the combination of exploiting the symmetry and removing the
square roots reduces computation time without degrading accuracy.
\end{abstract}



\section{Introduction}
\label{sec1_intro}

This project began with an observation. Among all the different types of
structured matrices that admit fast algorithms for computing eigenvalues,
symmetric tridiagonal matrices and unitary upper Hessenberg matrices possess a
number of very similar algorithms. For both classes there exist fast versions
of Francis' implicitly-shifted QR algorithm \cite{dubrulle1968, gragg1986} and
fast divide-and-conquer algorithms \cite{cuppen1980, gragg1990}. There are also
Krylov methods based on short recurrences \cite{lanczos1950, bohnhorst1993},
bisection methods \cite{parlett1998, bunsegerstner1984}, and Jacobi methods
\cite{jacobi1846, goldstine1959}, though, in general, these apply to full
matrices.

In the case of the QR algorithm, the similarities between these two classes of
matrices run even deeper. For both classes, the recursions used in each QR
iteration are not unique, leading to multiple versions of the ``same''
algorithm. In the case of unitary matrices, many variants have been proposed,
though not all of them are stable \cite{gragg1986, ammar1986, gragg1997,
stewart1998, gragg1999, stewart2001, wang2002, stewart2006, aurentz2015a}.
Additionally, for both classes of matrices there exist square-root-free
versions, which avoid computing the square roots needed to normalize the
similarity transformations. When done correctly, removing the costly
square roots gives an even more efficient algorithm for computing eigenvalues,
without sacrificing backward stability. For symmetric tridiagonal matrices, the
first square-root-free algorithm was proposed by Ortega and Kaiser in 1963
\cite{ortega1963}, and later improved by Pal, Walker, and Kahan in 1968
\cite[Section 8.15]{parlett1998}. Further versions were introduced and their
stability analyzed in \cite{feng1991}.

Developing square-root-free algorithms for unitary upper Hessenberg matrices is
more complicated than for symmetric tridiagonal matrices, since the
transforming matrices are necessarily complex. In the original unitary QR
algorithm by Gragg \cite{gragg1986}, cancellation errors due to the necessary
complex arithmetic made the algorithm unstable. In \cite{gragg1997,gragg1999},
Gragg introduced an alternative formulation of his original algorithm and
showed experimentally that it was stable. This same reformulation also made it
possible to develop a square-root-free unitary QR algorithm, similar to the
symmetric case. In \cite{stewart1998,stewart2001,stewart2006}, Stewart reviewed
several variants of the QR algorithm for unitary upper Hessenberg matrices and
gave a detailed stability analysis of a stabilized recurrence-based
formulation, including square-root-free variants.

In \cite{aurentz2015a}, the authors unified many of the existing unitary QR
algorithms by interpreting them as unitary core-chasing algorithms
\cite{aurentz2018}. This point of view allows the same algorithms to be derived
from a strictly matrix-based analysis, as opposed to one based on recurrence
relations for orthogonal polynomials on the unit circle. It also allows one to
show that, in exact arithmetic, many of the seemingly different variants of the
QR algorithm for unitary upper Hessenberg-like matrices are identical. The
matrix-based analysis further aids in simplifying the proof of backward
stability. However, the additional symmetries needed to derive a
square-root-free unitary core-chasing algorithm were not identified in
\cite{aurentz2015a}, and the question of structured backward stability was not
addressed.

This paper addresses these details, starting in Section~\ref{sec2_uqr}, where we
show that the core-chasing formulation of the unitary QR algorithm contains a
previously unnoticed symmetry. This symmetry is generated by the first fusion
in a single-shift core-chasing step and is preserved, in exact arithmetic, by
a symmetric turnover. We also discuss its relation with the QR algorithm for
symmetric tridiagonal matrices.

In Section~\ref{sec3_nearly} we study what happens when the core
transformations are only nearly-unitary, as occurs when using floating-point
arithmetic. In this setting, the na\"ive symmetric turnover formulas can be
unstable, and one must use a modified formula analogous to Gragg's stabilized
recurrence. We prove that the resulting symmetric turnover is nearly exact for
nearly-unitary data. 

In Section~\ref{sec4_rfuqr} we show how the symmetry in the turnover can be
used to derive a square-root-free unitary core-chasing algorithm and how to
preserve the nearly-unitariness of the core transformations without computing
additional square roots.  

In Section~\ref{sec5_stability} we analyze the floating-point implementation.
We prove that both the symmetric and square-root-free turnovers are stable in a
forward sense for nearly-unitary floating-point input, and then show how this
forward error can be converted into a backward error result for the turnover.
We also show that a strict structured backward error result is impossible in
general.

Finally, in Section~\ref{sec6_experiments} we present numerical experiments
that demonstrate the practical effects of removing square roots, and compare
the symmetric and square-root-free algorithms with LAPACK and the unitary
core-chasing algorithm in \cite{aurentz2015a}.

\section{The Unitary Core-Chasing Algorithm}
\label{sec2_uqr}

The key to the efficiency of the core-chasing algorithms for unitary
upper Hessenberg matrices in \cite{aurentz2015a} is to first factor the matrix
into a product of elementary matrices called \emph{core transformations}.

\subsection{Core transformations}
The core transformation $U_i$ is an $n \times n$ unitary matrix which is equal
to the identity matrix everywhere except for the submatrix at the intersection
of rows and columns $i$ and $i+1$. Apart from a few exceptions, in this paper
we will use core transformations that are complex Givens' rotations, whose
$(i+1,i)$ entry is real and nonnegative. In this case, the core transformation
$U_i$ can be parameterized by a complex number $u_i$ and a nonnegative real
number $v_i$ such that $|u_i|^2 + v_i^2 = 1$,
$$ 
U_i = 
\begin{bmatrix}
\ddots &   &     &                & &              \\[-0.2em]   
       & 1 &     &                & &              \\[-0.6em]   
       &   & u_i & -v_i           & &              \\
       &   & v_i & \phantom{-{}}\uB_i & &          \\[-0.2em]
       &   &     &                & 1 &            \\[-0.6em]
       &   &     &                &   & \ddots
\end{bmatrix}.
$$
Since working with core transformations often involves working directly with
their nontrivial diagonal blocks, we will use the notation $[*]$ to denote any
submatrices needed to perform the operations. For example, we use the
notation $[U_i]$ to signify the $2 \times 2$ submatrix at the intersection of rows
and columns $i$ and $i+1$,
$$
[U_i] = 
\begin{bmatrix}
u_i & -v_i               \\
v_i & \phantom{-{}}\uB_i
\end{bmatrix}.
$$
Similarly, we use $[U_i U_{i+1}]$ to denote the nontrivial $3 \times 3$
diagonal block of the product,
$$
[U_i U_{i+1}] = 
\begin{bmatrix}
u_i & -v_i               &  \\
v_i & \phantom{-{}}\uB_i &  \\
    &                    & 1
\end{bmatrix}
\begin{bmatrix}
1 &     &                    \\
  & u_{i+1} & -v_{i+1}               \\
  & v_{i+1} & \phantom{-{}}\uB_{i+1}
\end{bmatrix}.
$$
A fundamental result, essential to all fast unitary QR algorithms, is that
every $n \times n$ unitary upper Hessenberg matrix $U$ can be factored into a
product of $n$ core transformations $\{U_i\}_{i=1}^n$ as 
$$
U = U_1 \ldots U_{n-1} U_n.
$$ 
The last core transformation $U_n$ is equal to the identity matrix everywhere
except for the last diagonal entry, and is not, strictly speaking, a
rotation 
$$ 
U_n = 
\begin{bmatrix}
\ddots &   &                  \\[-0.4em]
       & 1 &                  \\[-0.4em]
       &   &  u_n  
\end{bmatrix}.
$$
This factorization implies that $U$ can be stored using $O(n)$ data, equivalent
to the linear storage needed for symmetric tridiagonal matrices. Gragg showed
using recurrence relations that, just as in the symmetric case, this
factorization can be used to perform a single step of Francis'
implicitly-shifted QR algorithm in $O(n)$ flops. Since this factorization is
the same one that is used in all core-chasing algorithms \cite{aurentz2018}, it
immediately admits a fast core-chasing algorithm for unitary upper Hessenberg
matrices.

\subsection{Unitary core-chasing}

Given a complex shift $\rho$, $|\rho| = 1$, an iteration of the unitary
core-chasing algorithm begins by constructing a core transformation $Q_1$ such
that $\alpha Q_1e_1 = (U - \rho I)e_1$, $\alpha \neq 0$, and performing the
similarity transformation $Q_1^* U Q_1$,
$$
Q_1^* U Q_1 = Q_1^* U_1 U_2 \ldots U_n Q_1 = Q_1^* U_1 U_2 Q_1 U_3 \ldots U_n.
$$
The last equality holds since $Q_i U_j = U_j Q_i$ whenever $|i-j| > 1$. 

The core transformation $Q_1$ is constructed so that its first column is
parallel to the first column of $U - \rho I$. Since $U-\rho I$ is
upper Hessenberg, its first column has only two nonzero entries, $(U-\rho I)e_1
= (U_1 - \rho I)e_1 = [u_1 - \rho, v_1,0,\ldots]^T$. Such a $Q_1$ can be
constructed as follows,
$$ 
[Q_1] = 
\begin{bmatrix}
 c_1 & -s_1               \\
 s_1 & \phantom{-{}}\cB_1
\end{bmatrix},
$$
where $c_1 = (u_1-\rho)/n_1$, $s_1 = v_1/n_1$ and $n_1 = \sqrt{|u_1-\rho|^2 +
v_1^2}$. The first similarity transformation is then simplified by
\emph{fusing} the product $Q_1^* U_1$ into a new core transformation $P_1$.
After applying this first similarity transform, the matrix $Q_1^* U Q_1$ is no
longer upper Hessenberg due to the factor $Q_1$ on the right, 
$$
Q_1^* U Q_1 = P_1 U_2 Q_1 U_3 \ldots U_n.
$$ 
The rest of the iteration consists of returning this product to
upper Hessenberg form through a sequence of turnovers and similarity
transformations. 

Given a product of three core transformations of the form
$A_i B_{i+1} C_i$, a \emph{turnover} is a refactorization into a product of
three new core transformations of the form $A_{i+1} B_i C_{i+1}$ such that, 
$$
A_i B_{i+1} C_i = A_{i+1} B_i C_{i+1}.
$$ 
When $A_i$, $B_{i+1}$ and $C_i$ are unitary, the turnover can always be
computed stably in $O(1)$ flops by computing the QR factorization of a $3
\times 3$ unitary matrix \cite{aurentz2018}.

The iteration continues by computing a turnover of the product $P_1 U_2 Q_1$,
$$
P_1 U_2 Q_1 = Q_2 \hU_1 P_2,
$$
and performing a similarity transformation with the computed matrix $Q_2$,
$$ 
Q_2^* Q_1^* U Q_1 Q_2 = \hU_1 P_2 U_3 Q_2 U_4 \ldots U_n.
$$
This process is continued a total of $n-1$ times, resulting in the following
product,
$$ 
Q_{n-1}^*\ldots Q_1^* U Q_1 \ldots Q_{n-1} = \hU_1 \ldots \hU_{n-2}P_{n-1} U_n
Q_{n-1}.
$$
Although $U_n$ is not a rotation, the product $P_{n-1} U_n Q_{n-1}$ can be
refactored via a turnover as $Q_n\hU_{n-1}P_n$, where $Q_n$ and $P_n$ have
the same form as $U_n$. Performing a similarity transformation with the matrix
$Q_n$ gives the following, 
$$ 
Q_{n}^* \ldots Q_1^* U Q_1 \ldots Q_{n} = \hU_1 \hU_2 \ldots \hU_{n-1} P_n Q_n. 
$$
Thus, a single iteration of the unitary core-chasing algorithm directly
computes the factorization for the resulting unitary upper Hessenberg matrix
$\hU$,
$$ 
Q^* U Q = \hU_1 \hU_2 \ldots \hU_{n-1} \hU_n = \hU, 
$$
where $Q = Q_1 \ldots Q_{n}$ and $\hU_n = P_n Q_n$.

\subsection{Symmetry in the unitary core-chasing algorithm}

The majority of the work in a single unitary core-chasing iteration comes from
computing the $O(n)$ turnovers. In the algorithm presented in
\cite{aurentz2015a} this is done via a QR factorization of a $3 \times 3$
unitary matrix. While this can be done stably using a single square root, it
requires explicitly forming most of the $3 \times 3$ submatrix and multiplying
these entries by two rotations. Here it will be shown that the arithmetic
needed to perform a turnover can be reduced by taking advantage of a previously
unknown symmetry present in the core transformations. This symmetry comes from
the special form of the initial transforming matrix $Q_1$ and is maintained
throughout the entire iteration. 

Recall that given a unimodular shift $\rho$, the core transformation $Q_1$ is
constructed so that its first column is parallel to the first column of $U_1-\rho
I$, 
$$ 
[Q_1] = 
\begin{bmatrix}
 c_1 & -s_1               \\
 s_1 & \phantom{-{}}\cB_1
\end{bmatrix}
=
\begin{bmatrix}
 (u_1-\rho)/n_1 & -v_1/n_1               \\
 v_1/n_1 & \phantom{-{}}(\uB_1-\rhoB)/n_1
\end{bmatrix},
$$
where $n_1 = \sqrt{|u_1-\rho|^2 + v_1^2}$. Thus, the parameters that generate
$Q_1$ are shifted and scaled versions of the parameters that generate $U_1$.
The initial similarity transform requires the fusion of the product $Q_1^*
U_1$. Given that both matrices are core transformations, computing the fusion
requires computing the product of their initial $2\times 2$ diagonal
blocks.

To see the structure in the initial fusion we first separate the diagonal
shifts in $[Q_1]$,
$$ 
[Q_1] = 
\frac{1}{n_1}
\left(
\begin{bmatrix}
 u_1 & -v_1 \\
 v_1 & \phantom{-{}}\uB_1
\end{bmatrix}
+
\begin{bmatrix}
 -\rho &            \\
  & -\rhoB
\end{bmatrix}
\right).
$$
Then, we compute the products of the initial $2 \times 2$ diagonal blocks
\begin{equation}\label{eq:fuse1}
\begin{split}
[Q_1^*U_1]
&= 
\frac{1}{n_1}
\left(
\begin{bmatrix}
 \phantom{-{}}\uB_1 & v_1\\ 
 -v_1                      & u_1 
\end{bmatrix}
+
\begin{bmatrix}
 -\rhoB &            \\
  & -\rho
\end{bmatrix}
\right)
\begin{bmatrix}
 u_1 & -v_1 \\
 v_1 & \phantom{-{}}\uB_1
\end{bmatrix}\\
&= 
\begin{bmatrix}
 -\rhoB &            \\
  & -\rho
\end{bmatrix}
\frac{1}{n_1}
\left(
\begin{bmatrix}
 -\rho &            \\
  & -\rhoB
\end{bmatrix}
+
\begin{bmatrix}
 u_1 & -v_1 \\
 v_1 & \phantom{-{}}\uB_1
\end{bmatrix}
\right)\\
&=
\begin{bmatrix}
 -\rhoB &            \\
  & -\rho
\end{bmatrix}
\begin{bmatrix}
 c_1 & -s_1 \\
 s_1 & \phantom{-{}}\cB_1
\end{bmatrix}
=  
\begin{bmatrix}
-\rhoB c_1 & \phantom{-{}}\rhoB s_1 \\
-\rho s_1 & -\rho\cB_1
\end{bmatrix}
=  [\tQ_1].
\end{split}
\end{equation}
Now we can rewrite the first similarity transformation and fusion as follows,
$$
Q_1^* U Q_1 = \tQ_1 U_2 Q_1 U_3 \ldots U_n.
$$
The matrix $\tQ_1$ is still a core transformation but the off-diagonal entries
of \([\tQ_1]\) are no longer strictly real. However, their phase is $-\rho$
and the connection to \( Q_1 \) remains via the parameters \(c_1\) and \(s_1\). 

\subsubsection{Symmetry in the turnover}

In order for this symmetry to be useful it must be preserved by the turnover.
If that were the case, there would exist core transformations $Q_2$, $\hU_1$
and $\tQ_2$ such that
$$
\tQ_1 U_2 Q_1 = Q_2 \hU_1 \tQ_2.
$$
Then, applying the similarity transformation defined by $Q_2$ would give the
next step in the iteration,
$$
Q_2^* Q_1^* U Q_1 Q_2 = \hU_1 \tQ_2 U_3 Q_2 U_4 \ldots U_n,
$$
and the product $\tQ_2 U_3 Q_2$ would have the same structure as the previous
turnover, thus preserving the symmetry. 

Unfortunately, the existence of the shift $\rho$ in $\tQ_1$ breaks the symmetry
just enough to require the introduction of an extra parameter.
Theorem~\ref{thm:symto} shows that with the addition of this new parameter, the
symmetry in the turnover is preserved throughout the entire core-chasing step.

\begin{theorem}[Symmetric Turnover]\label{thm:symto}
Given unimodular complex numbers $\rho$ and $\gamma$, complex numbers
$u$ and $c$, and nonnegative real numbers $v$ and $s$ such
that $|u|^2 + v^2 = 1$ and $|c|^2 + s^2 = 1$, define the
following, 
\begin{equation*}
\begin{array}{l}
\begin{array}{lll}
p &=& -\rho(c + \gammaB\cB u), \\
n &=& \sqrt{|p|^2 + v^2},
\end{array} \\
\begin{array}{llll}
\mathrm{if}\ n \neq 0: & & &\\
& \hc &=& p/n,\\
& \hs &=& v/n,\\
& \hu &=& \rhoB(s^2u - \gamma c^2),\\
& \hv &=& s n\\
\mathrm{else}: & & &\\
& \hc &=& 0,\\
& \hs &=& 1,\\
& \hu &=& \rhoB u,\\
& \hv &=& 0,\\
\end{array} \\
\begin{array}{lll}
\hgamma &=& -\rhoB\gamma.
\end{array}
\end{array}
\end{equation*}
Then, the following products are equivalent,
\[
\begin{bmatrix}
-\rhoB \gamma c & \phantom{-{}}\rhoB s  & \\
-\rho s & -\rho\gammaB \cB & \\
     &                & 1 
\end{bmatrix}
\begin{bmatrix}
 1    &                & \\  
 &  u & -v           \\
 &  v & \phantom{-{}}\uB
\end{bmatrix}
\begin{bmatrix}
 c & -s           & \\
 s & \phantom{-{}}\cB \\
     &                & 1 
\end{bmatrix}
= 
\begin{bmatrix}
 1    &                & \\  
 & \hc & -\hs           \\
 & \hs & \phantom{-{}}\hcB
\end{bmatrix}
\begin{bmatrix}
 \hu & -\hv           & \\
 \hv & \phantom{-{}}\huB \\
     &                & 1 
\end{bmatrix}
\begin{bmatrix}
 1    &                & \\  
 & -\rhoB\hgamma\hc & \phantom{-{}}\rhoB\hs           \\
 &  -\rho \hs & -\rho\hgammaB\hcB
\end{bmatrix}.
\]
\end{theorem}

\begin{proof}
Define \(H\) and \(\hH\) as follows
\[
H =
\begin{bmatrix}
-\rhoB \gamma c & \phantom{-{}}\rhoB s  & \\
-\rho s & -\rho\gammaB \cB & \\
     &                & 1 
\end{bmatrix}
\begin{bmatrix}
 1    &                & \\  
 &  u & -v           \\
 &  v & \phantom{-{}}\uB
\end{bmatrix}
\begin{bmatrix}
 c & -s           & \\
 s & \phantom{-{}}\cB \\
     &                & 1 
\end{bmatrix}
=
\begin{bmatrix}
\rhoB(s^2u-\gamma c^2) & \rhoB \gamma(c+\gammaB\cB u)s & -\rhoB vs\\
-\rho (c+\gammaB\cB u)s & \rho(s^2 - \gammaB\cB^2u) & \rho v\gammaB\cB \\
vs & v\cB & \uB
\end{bmatrix}.
\]
\[ 
\hH = 
\begin{bmatrix}
 1    &                & \\  
 & \hc & -\hs           \\
 & \hs & \phantom{-{}}\hcB
\end{bmatrix}
\begin{bmatrix}
 \hu & -\hv           & \\
 \hv & \phantom{-{}}\huB \\
     &                & 1 
\end{bmatrix}
\begin{bmatrix}
 1    &                & \\  
 & -\rhoB\hgamma\hc & \phantom{-{}}\rhoB\hs           \\
 &  -\rho \hs & -\rho\hgammaB\hcB
\end{bmatrix}
=
\begin{bmatrix}
\hu
&
\rhoB\hgamma\hc\hv
&
-\rhoB\hs\hv
\\
\hc\hv
&
\rho\hs^2-\rhoB\hgamma\hc^2\huB
&
\rhoB\hc\hs\huB+\rho\hgammaB\hcB\hs
\\
\hs\hv
&
-\rhoB\hgamma\hc\hs\huB-\rho\hcB\hs
&
\rhoB\hs^2\huB-\rho\hgammaB\hcB^2
\end{bmatrix}.
\]
It is enough to show that \(H=\widehat H\). Using \( p=-\rho(c+\gammaB\cB u)
\) and \( \hgamma=-\rhoB\gamma, \) together with \(|\rho|=|\gamma|=1\) and \(s^2
= 1 - c\cB\), \(H\) becomes
\[
H= 
\begin{bmatrix}
\rhoB(s^2u-\gamma c^2) & \rhoB \hgamma ps & -\rhoB vs\\
ps & \rho(s^2 - \gammaB\cB^2u) & -v\cB\hgammaB \\
vs & v\cB & \uB
\end{bmatrix}.
\]
In the case \(n\neq0\), define \(\hQ\) and \(\htQ\) as follows,
\[ 
\hQ = 
\begin{bmatrix}
 1    &                & \\  
 & \hc & -\hs           \\
 & \hs & \phantom{-{}}\hcB
\end{bmatrix},
\qquad
\htQ = 
\begin{bmatrix}
 1    &                & \\  
 & -\rhoB\hgamma\hc & \phantom{-{}}\rhoB\hs \\
 &  -\rho \hs & -\rho\hgammaB\hcB
\end{bmatrix}.
\]
The matrices \(\hQ\) and \(\htQ\) are designed to create zeros in the
$(3,1)$ and $(1,3)$ entries of $H$, respectively, so that 
\[
K = \hQ^* H \htQ^* = 
\begin{bmatrix}
\hu & -\hv &  \\
\hv & \phantom{-{}}* & *  \\
    & \phantom{-{}}* & * 
\end{bmatrix}.
\]
Assume that $s \neq 0$. Then $\hv \neq 0$ and, since $K$ is unitary,
$K(3,2) = K(2,3) = 0$. This implies that $K(2,2) = \huB$; since $\det(H) =
\det(K) = 1$, we also have $K(3,3) = 1$. Therefore,
\[
K = 
\begin{bmatrix}
 \hu & -\hv           & \\
 \hv & \phantom{-{}}\huB \\
     &                & 1 
\end{bmatrix},
\]
and 
\[
H = 
\begin{bmatrix}
 1    &                & \\  
 & \hc & -\hs           \\
 & \hs & \phantom{-{}}\hcB
\end{bmatrix}
\begin{bmatrix}
 \hu & -\hv           & \\
 \hv & \phantom{-{}}\huB \\
     &                & 1 
\end{bmatrix}
\begin{bmatrix}
 1    &                & \\  
 & -\rhoB\hgamma\hc & \phantom{-{}}\rhoB\hs           \\
 &  -\rho \hs & -\rho\hgammaB\hcB
\end{bmatrix}
= \hH.
\]
For $s = 0$, notice that for $n \neq 0$, the entries of $H$ and $\hH$ are
continuous in $s$ as long as the phase of $c$ is unchanged,
\[
\lim_{s \rightarrow 0} |c| = 1, \qquad \frac{c}{|c|} = \textrm{constant}.
\] 
Since $H-\hH = 0$ for $s \neq 0$, by continuity, $\lim_{s \rightarrow 0} (H-\hH) = 0$.  

In the case \(n=0\), use $s^2 = 1 - c\cB$ to rewrite $H$ as,
\[
H= 
\begin{bmatrix}
\rhoB(u-\hgamma c p) & \rhoB \hgamma ps & -\rhoB vs\\
ps & \rho + p\cB  & -v\cB\hgammaB \\
vs & v\cB & \uB
\end{bmatrix}.
\]
Now $n = 0$ implies that \(p = v = 0\) and the condition \(|u|^2+v^2=1\)
implies \(|u|=1\), so the expression for \(H\) reduces to
\[
H=
\begin{bmatrix}
\rhoB u& & \\
 &\rho& \\
 & &\uB
\end{bmatrix}.
\]
By construction
\[
\widehat c=0,
\qquad
\widehat s=1,
\qquad
\widehat u=\rhoB u,
\qquad
\widehat v=0,
\]
which gives
\[
\widehat H
=
\begin{bmatrix}
\hu & & \\
 &\rho & \\
 & & \rhoB\huB
\end{bmatrix}
=
\begin{bmatrix}
\rhoB u& & \\
 &\rho& \\
 & &\uB
\end{bmatrix}
= H.
\]
This completes the proof. 
\end{proof}

Theorem~\ref{thm:symto} shows that the product $\tQ_1 U_2 Q_1$ can be
refactored into a new set of core transformations, 
$$
\tQ_1 U_2 Q_1 = Q_2 \hU_1 \tQ_2,
$$
where $\tQ_2$ is the same as $Q_2$ except for the unimodular factors $\rho$ and
$\hgamma$, and computing a similarity transform with the matrix $Q_2$ preserves
the structure of the initial similarity transform, only one row lower and one
column to the right, 
$$
Q_2^* Q_1^* U Q_1 Q_2 = \hU_1 \tQ_2 U_3 Q_2 U_4 \ldots U_n.
$$
Repeating this process $n-1$ times completes a single step of the unitary
core-chasing algorithm,
$$ 
Q^* U Q = Q_{n}^* \ldots Q_1^* U Q_1 \ldots Q_{n} = \hU_1 \hU_2 \ldots \hU_{n-1}
\hU_n = \hU. 
$$

\subsection{Connection to symmetric tridiagonal QR}
The formulas in Theorem~\ref{thm:symto} are possible thanks to the special
structure of unitary upper Hessenberg matrices. Symmetric tridiagonal matrices
also possess a similar structure that allows for the efficient implementation of
Francis' single-shift QR algorithm. Even though the algorithm is not a
core-chasing algorithm in the sense of \cite{aurentz2018}, the structure is
remarkably similar. As an illustrative example, consider the following $3 \times
3$ symmetric tridiagonal matrix, 
\[
T = 
\begin{bmatrix}
a_1 & b_1 &    \\
b_1 & a_2 & b_2\\
    & b_2 & a_3
\end{bmatrix}.
\]
Given a real shift $\rho$, the QR iteration begins by applying the similarity
transform $Q_1^T T Q_1$ where $Q_1$ is the following core transformation,
\[
Q_1 = 
\begin{bmatrix}
c_1 & -s_1 &    \\
s_1 & \phantom{-{}}c_1 & \\
    &  & 1
\end{bmatrix},
\]
with $n_1 = \sqrt{(a_1-\rho)^2 + b_1^2}$, $c_1 = (a_1-\rho)/n_1$ and $s_1 =
b_1/n_1$. The effect of this similarity transform is to create a nonzero 
\emph{bulge} in the $(3,1)$ and $(1,3)$ entries of $Q_1^T TQ_1$,  
\[
Q_1^T TQ_1 = 
\begin{bmatrix}
* & * & + \\
* & * & * \\
+ & * & * 
\end{bmatrix}.
\]
The algorithm proceeds by constructing a core transformation $Q_2$ that 
removes the bulge and returns the matrix to tridiagonal form,
\[
Q_2^TQ_1^T TQ_1Q_2 = 
\begin{bmatrix}
\widehat{a}_1 & \widehat{b}_1 &               \\
\widehat{b}_1 & \widehat{a}_2 & \widehat{b}_2 \\
              & \widehat{b}_2 & \widehat{a}_3
\end{bmatrix}
= \widehat{T}.
\]
Moving $Q_2$ over to $\widehat{T}$ gives the equivalent of a turnover for symmetric
tridiagonal matrices,
\[
Q_1^T TQ_1 = Q_2\widehat{T}Q_2^T.
\]
The following formulas can be used to compute the entries of $Q_2$ and the
first row and column of $\widehat{T}$ in exact arithmetic\footnote{See
\cite{parlett1998} for numerically stable alternatives.},
\[
\begin{array}{lll}
p_2 &=& c_1(a_2-\rho)-s_1b_1,\\
n_2 &=& \sqrt{p_2^2+b_2^2},\\
c_2 &=& p_2/n_2,\\
s_2 &=& b_2/n_2,\\
\widehat a_1 &=& c_1^2a_1+2c_1s_1b_1+s_1^2a_2,\\
\widehat b_1 &=& s_1 n_2.
\end{array}
\]
As we can see, these formulas are similar in structure to those used in the
symmetric turnover for unitary matrices, with the main differences being the
need for complex arithmetic and the phase-tracking variable $\gamma$ in the
unitary case. 

\section{Nearly-Unitary Matrices}
\label{sec3_nearly}

Although the results of Theorem~\ref{thm:symto} allow us to compute the
turnover efficiently in exact arithmetic, the current formulas are not
stable and do not permit the development of a square-root-free symmetric
turnover.
The instabilities come from the fact that the floating-point representations of
the parameters $u$ and $v$ can no longer be guaranteed to satisfy $|u|^2 + v^2
= 1$ exactly, and from the possibility of cancellation when computing the
turnover in floating-point arithmetic. 

Similar numerical instabilities were present in Gragg's first unitary QR
algorithm \cite{gragg1986} based on recurrence relations for Szeg\H{o}
polynomials. In \cite{gragg1999}, Gragg introduced a new change of variables
for these relations that appeared to remove these instabilities and also
allowed for a square-root-free implementation, though no proof was given. A
proof of stability of these recurrence-based algorithms was given by Stewart in
\cite{stewart2006}. 

In this and the following sections, we will show that a similar change of
variables to that found in \cite{gragg1999} can be used to stabilize the
symmetric turnover and to develop a square-root-free implementation. We will
also show that these versions of the symmetric turnover are stable when applied
to matrices that are only nearly-unitary in floating-point arithmetic.  

\subsection{A change of variables}
In its current form, the symmetric turnover is not stable and does not allow
for a square-root-free implementation. The main impediment is the formula for
$p$,
\[
p = -\rho(c + \gammaB\cB u).
\] 
When computed in floating-point arithmetic, the product $\gammaB \cB u$
followed by the addition with $c$ is prone to cancellation as the following
example illustrates.

\begin{example}[Instability of the Symmetric Turnover Before Changing Variables]
\label{ex:p-instability}
Let $\rerr=2^{-53}$ denote the unit roundoff for double precision and let
$\fl(\cdot)$ denote rounding to nearest. Define

$$
a=\fl\left(\frac{\sqrt{3}}{2}\right)
  =0.86602540378443860
$$

and consider the floating-point input parameters

$$
u=i,\qquad
v=0,\qquad
c=\gamma=a-\frac{1}{2}i,\qquad
s=0,\qquad
\rho=1.
$$

Since

$$
|u|^2+v^2=1
$$

and

$$
|c|^2=|\gamma|^2
=1-1.1102230246251565\times10^{-16}
=1-\rerr,
$$

the input parameters are unitary to working precision.  The corresponding input
product is diagonal and satisfies

$$
H=\tQ UQ
=
\begin{bmatrix}
 i & 0 & 0\\
 0 & 1 & 0\\
 0 & 0 & -i
\end{bmatrix}
+O(\rerr).
$$

%

Now compute the symmetric turnover in floating-point arithmetic using the
formulas in Theorem~\ref{thm:symto}.  The first intermediate quantity is

$$
p=-\rho\left(c+\gammaB\cB u\right).
$$

Since $\gamma=c$, floating-point evaluation gives

$$
\gamma c
=
\left(\frac12-\rerr\right)-ai
\qquad
\mathrm{and}
\qquad
\gammaB\cB u
=
-a+\left(\frac12-\rerr\right)i.
$$

Since

$$
c=a-\frac12 i,
\qquad
p=\rerr i.
$$

It follows that

$$
n=\sqrt{|p|^2+v^2}=\rerr,
\qquad
\hc=\frac{p}{n}=i,
\qquad
\hs=\frac{v}{n}=0.
$$

The remaining output parameters are

$$
\hu=\rhoB(s^2u-\gamma c^2)
   =-\gamma c^2
   =i+O(\rerr),
\qquad
\hv=sn=0,
\qquad
\hgamma=-\rhoB\gamma=-\gamma.
$$

The corresponding computed output product satisfies

$$
\hH=\hQ\hU\htQ
=
\begin{bmatrix}
i & 0 & 0\\
0 & \displaystyle\frac12+\frac{\sqrt3}{2}i & 0\\
0 & 0 & \displaystyle-\frac{\sqrt3}{2}-\frac12i
\end{bmatrix}
+O(\rerr).
$$

Consequently,

$$
\hH-H
=
\begin{bmatrix}
0&0&0\\
0&
-\displaystyle\frac12+\frac{\sqrt3}{2}i
&0\\
0&0&
-\displaystyle\frac{\sqrt3}{2}+\frac12i
\end{bmatrix}
+O(\rerr).
$$

Both nonzero diagonal entries of the leading term have modulus one, and hence

$$
\|\hH-H\|_2=1+O(\rerr).
$$

%
%

Thus, even though the input parameters represent core transformations that are
unitary to working precision, evaluation of the symmetric turnover in
Theorem~\ref{thm:symto} can produce an $O(1)$ error in the resulting product.
\end{example}

In addition to the possibility of cancellation, the square-root-free turnover
that we develop in the next section requires computing $\hc^2$ directly from
$c^2$, something the current formula for $p$ does not easily allow. 

Fortunately, both of these problems are avoidable if we make a few changes of
variables. The first change is to express $c$ in polar form, $c = |c|\sigma$.
Applying this change to $p$ and factoring out $\sigma$ gives the following,
\[
p = -\rho|c|\sigma(1 + \gammaB\sigmaB^2 u).
\] 
We then introduce the additional variables $t$ and $z$, defined by
\[
t = \gammaB\sigmaB^2 u, \qquad
z = 1+t,
\]
which allows us to rewrite $p$ as,
\[
p = -\rho|c|\sigma z.
\] 
The expression for $t$ uses only multiplication, which aids stability and
the development of a square-root-free version. The formula for $z$, however, 
can still suffer cancellation when $t \approx -1$. Using another idea from
\cite{gragg1999}, we can derive an alternative formula for $z$ by multiplying by
$(1-\tB)$,
\[
(1-\tB)z = (1-\tB)(1+t) = 1-|t|^2 + (t-\tB) = v^2 + 2\im(t) i, 
\]
where we used $1-|t|^2 = 1-|u|^2 = v^2$. This gives the following alternatives
for computing $z$,
\[
z = 
\begin{cases}
1+t, & \re(t) \geq 0,\\
\frac{v^2 + 2\im(t)i}{1-\tB}, & \re(t) < 0.
\end{cases}
\]
The branch conditions are chosen so that there is never cancellation in $(1+t)$
or $(1-\tB)$. Gragg suggested without proof that this alternative formulation
for $z$ was stable because it uses $v$ in the case $t \approx -1$ which implies
$|u| \approx 1$ and $v \approx 0$. As we will see later, the reason for the
improved stability is more subtle.

For now we continue rewriting the symmetric turnover formulas using $t$, $z$ and
the polar decompositions $c = |c|\sigma$ and $\hc = |\hc|\hsigma$: 
\[
\begin{array}{lll}
n &=& \sqrt{|p|^2 + v^2} = \sqrt{|c|^2|z|^2 + v^2},\\
|\hc| &=& |p|/|n| = |c||z|/n,\\
\hsigma &=& \hc/|\hc| = p/|p| = -\rho\sigma z/|z|,\\
\hs &=& v/n,\\
\hu &=& \rhoB(s^2u - \gamma\sigma^2 |c|^2),\\
\hv &=& s n,\\
\hgamma &=& -\rhoB \gamma.
\end{array}
\]
Notice that $p$ is no longer explicitly computed, since we can update $|\hc|$
and $\hsigma$ directly using $z$ and the other input variables. 

Let $\tQ$, $U$ and $Q$ be inputs of the symmetric turnover, after applying the 
preceding change of variables,
\[
\tQ = 
\begin{bmatrix}
-\rhoB \gamma c \sigma & \phantom{-{}}\rhoB s  & \\
-\rho s & -\rho\gammaB c\sigmaB & \\
     &                & 1 
\end{bmatrix},
\qquad
U = 
\begin{bmatrix}
 1    &                & \\  
 &  u & -v           \\
 &  v & \phantom{-{}}\uB
\end{bmatrix},
\qquad
Q = 
\begin{bmatrix}
 c\sigma & -s           & \\
 s & \phantom{-{}}c\sigmaB \\
     &                & 1 
\end{bmatrix},
\]
and $\hQ$, $\hU$ and $\htQ$ be the corresponding outputs,
\[
\hQ = 
\begin{bmatrix}
 1    &                & \\  
 & \hc\hsigma & -\hs           \\
 & \hs & \phantom{-{}}\hc\hsigmaB
\end{bmatrix},
\qquad 
\hU = 
\begin{bmatrix}
 \hu & -\hv           & \\
 \hv & \phantom{-{}}\huB \\
     &                & 1 
\end{bmatrix},
\qquad
\htQ = 
\begin{bmatrix}
 1    &                & \\  
 & -\rhoB\hgamma\hc\hsigma & \phantom{-{}}\rhoB\hs           \\
 &  -\rho \hs & -\rho\hgammaB\hc\hsigmaB
\end{bmatrix},
\]
where the values of $|c|$ and $|\hc|$ have been replaced by $c$ and $\hc$,
respectively.  Corollary~\ref{cor:symto} shows how to compute the symmetric
turnover using the revised formulas.  
\begin{corollary}[Symmetric Turnover]\label{cor:symto}
Given unimodular complex numbers $\rho$, $\gamma$ and $\sigma$, a complex number $u$,
and nonnegative real numbers $v$, $c$ and $s$ such that $|u|^2 + v^2 = 1$ and
$c^2 + s^2 = 1$, define the following, 
\begin{equation*}
\begin{array}{l}
\begin{array}{lll}
t &=& \gammaB\sigmaB^2 u,
\end{array} \\
\begin{array}{lll}
z &=& 
\begin{cases}
1+t, & \re(t) \geq 0,\\
\frac{v^2 + 2\im(t)i}{1-\tB}, & \re(t) < 0,
\end{cases}
\end{array} \\
\begin{array}{lll}
n &=& \sqrt{c^2|z|^2 + v^2},
\end{array} \\
\begin{array}{llll}
\mathrm{if}\ n \neq 0: & & &\\
& \hc &=& c|z|/n,\\
& \hsigma &=& -\rho \sigma z/|z|,\\
& \hs &=& v/n,\\
& \hu &=& \rhoB(s^2u - \gamma\sigma^2c^2),\\
& \hv &=& s n\\
\mathrm{else}: & & &\\
& \hc &=& 0,\\
& \hsigma &=& 1,\\
& \hs &=& 1,\\
& \hu &=& \rhoB u,\\
& \hv &=& 0,\\
\end{array} \\
\begin{array}{lll}
\hgamma &=& -\rhoB\gamma.
\end{array}
\end{array}
\end{equation*}
Then, 
\[
\tQ U Q = \hQ \hU \htQ.
\]
\end{corollary}
\begin{proof}
The principal difference with the symmetric turnover formulas in
Theorem~\ref{thm:symto} is the addition of the variables $\sigma$ and
$\hsigma$. In the $n \neq 0$ branch, the definition of $\hsigma$ requires that
$z \neq 0$. For unitary inputs both formulas for $z$ give the same result.
Using the $\re(t) \geq 0$ branch, notice that $z = 0 \implies |t| = |u| = 1
\implies v = 0$, thus $n = 0$. Therefore, $n \neq 0 \implies z \neq 0$. When
$n=0$, $\hc = 0$ and $\hsigma$ is not uniquely defined, so it is set equal to 1
for simplicity. The rest of the proof follows directly from
Theorem~\ref{thm:symto} and the previously discussed changes of variables. 
\end{proof}

\subsection{Nearly-unitary matrices}
Just as in the symmetric tridiagonal case, the symmetry in the unitary
core-chasing algorithm comes directly from the underlying unitary structure of
the core transformations. Unlike the symmetric case, representing matrices that
are exactly unitary is typically not possible using floating-point numbers.
Take for example $u = v = 1/\sqrt{2}$. Define $\fl(u)$ and $\fl(v)$ to be their
nearest floating-point number approximations, then in exact arithmetic $u^2 +
v^2 = 1$, but $\fl(u)^2 + \fl(v)^2 \neq 1$, since $1/\sqrt{2}$ is irrational. 

Fortunately, even though $\fl(u)^2 + \fl(v)^2 \neq 1$, there exists $-1 <
\epsilon < 1$ such that $\fl(u)^2 + \fl(v)^2 = 1 + \epsilon$. If $\epsilon$ is
not too large, say on the order of machine precision, we can think of the
parameters $\fl(u)$ and $\fl(v)$ as representing a core transformation that is
\emph{nearly-unitary}. 

We say that the core transformation $U$ is nearly-unitary if there exists $0
\leq \epsilon < 1$ such that
\[
\| U^* U - I\| = O(\epsilon),
\]
and, equivalently, that the corresponding parameters $u$ and $v$ are nearly-unitary if
\[
\euv = |u|^2 + v^2 - 1,
\qquad
\left| \euv \right| = O(\epsilon).
\]
We refer to $\euv$ as the \emph{unitary defect} for the parameters $u$ and $v$.
\footnote{A similar concept is defined in \cite[Section 5.1]{GolubVanLoan2013}.}

Since the symmetric turnover is only defined for unitary matrices, we need to
analyze what happens when the input core transformations are only
nearly-unitary. While the analysis that follows is true for any modest value of
$\epsilon$, in practice $\epsilon$ is on the order of machine precision and the
constant in the $O(\epsilon)$ term is on the order of $10^2$.  

Fortunately, as the following theorem
demonstrates, if the inputs to the symmetric turnover are nearly-unitary, then
the outputs are also nearly-unitary and the products $\tQ U Q$ and $\hQ \hU
\htQ$ are close. 

\begin{theorem}[Exactness of Symmetric Turnover for Nearly-Unitary
Inputs]\label{thm:nearly}
Assume that the inputs of Corollary~\ref{cor:symto} are nearly-unitary for $0
\leq \epsilon < 1$. Then, for $\epsilon$ sufficiently small, the corresponding
outputs are also nearly-unitary with respect to the same $\epsilon$ and 
\[
\left\|\tQ U Q - \hQ \hU \htQ \right\| = O(\epsilon).
\]
\end{theorem}
\begin{proof}
Since the inputs are nearly-unitary, their unitary defects are all real and
$O(\epsilon)$,
\[
|u|^2 + v^2 = 1 + \euv,
\quad
c^2 + s^2 = 1 + \ecs,
\quad
|\sigma|^2 = 1 + \esigma,
\quad
|\gamma|^2 = 1 + \egamma,
\quad
|\rho|^2 = 1 + \erho.
\]
For $\epsilon$ sufficiently small, the unitary defects are all strictly less
than 1 in absolute value, so we can project the nearly-unitary inputs onto
nearby inputs that represent exactly unitary core transformations as follows,
\[
\uu = \frac{u}{\sqrt{1+\euv}},
\quad
\uv = \frac{v}{\sqrt{1+\euv}},
\quad
\uc = \frac{c}{\sqrt{1+\ecs}},
\quad
\us = \frac{s}{\sqrt{1+\ecs}},
\]
\[
\usigma = \frac{\sigma}{\sqrt{1+\esigma}},
\quad
\ugamma = \frac{\gamma}{\sqrt{1+\egamma}},
\quad
\urho = \frac{\rho}{\sqrt{1+\erho}}.
\]
The same $\epsilon$ can be taken sufficiently small such that the projected
inputs also satisfy
\[
\ux = x\ope,
\quad
x = u,v,c,s,\sigma,\gamma,\rho.
\]
Now let $\utQ$, $\uU$ and $\uQ$ be the input core transformations to the
symmetric turnover defined in terms of the projected parameters $\uu$, $\uv$,
$\uc$, $\us$, $\usigma$, $\ugamma$ and $\urho$. Then, it follows directly that 
\[
\left\|\tQ U Q - \utQ\,\uU\,\uQ \right\| = O(\epsilon).
\]

Since $\uu$, $\uv$, $\uc$, $\us$, $\usigma$, $\ugamma$ and $\urho$ represent
core transformations that are exactly unitary, applying the symmetric turnover
from Corollary~\ref{cor:symto} results in outputs $\uhu$, $\uhv$, $\uhc$,
$\uhs$, $\uhsigma$ and $\uhgamma$ that correspond to the exactly unitary
core transformations $\uhQ$, $\uhU$ and $\uhtQ$ which satisfy
\[
\utQ\,\uU\,\uQ = \uhQ\,\uhU\,\uhtQ.
\]

To finish the proof we must show that the outputs of applying the symmetric
turnover to the nearly-unitary inputs are also nearly-unitary, and that the
corresponding core transformations $\hQ$, $\hU$ and $\htQ$ satisfy
\[
\left\|\uhQ\,\uhU\,\uhtQ - \hQ\hU\htQ \right\| = O(\epsilon).
\]
To show this, we apply the symmetric turnover from Corollary~\ref{cor:symto} to
the nearly-unitary inputs and track how the relative errors from projecting
them to the exactly unitary inputs propagate. The first step is to compute $t$, 
\[
t = \gammaB \sigmaB^2 u = \underline{\smash{\overline{\gamma}}}
\overline{\usigma}^2 \uu \ope = \ut\ope.
\] 
This says that when $\epsilon$ is sufficiently small, the relative errors
introduced from projecting the nearly-unitary inputs translate into small
relative errors in the parameter $\ut$ computed using the exactly unitary
inputs. What is not shown explicitly is that this relative error is real and
positive. That is because the projections were performed using real, positive
scalars. Since $t$ and $\ut$ differ by a real positive mulitple, the sign of
$\re(t)$ and $\re(\ut)$ are the same, which preserves the branch structure.
When $\re(t) \geq 0$ we have that
\[
z = 1 + t = 1 + \ut\ope = (1+\ut)\ope = \uz\ope,
\]
where we used the fact that there is no cancellation in $1+\ut$. When 
$\re(t) < 0$ we have
\[
z = \frac{v^2 + 2\im(t) i}{1-\tB} = \frac{\uv^2\ope + 2\im(\ut) i\ope}{1-\utB\ope}
= \uz\ope,
\]
where we used the fact that $\uv^2$ and $2\im(\ut) i$ are orthogonal and that
there is no cancellation in $1-\utB$. Although the relative errors from projecting
the inputs were real, the relative error between $z$ and $\uz$ is, in general, complex. 

The last internal parameter is $n$,
\[
n = \sqrt{c^2|z|^2 + v^2} = \sqrt{\uc^2|\uz|^2\ope + \uv^2\ope} = \un\ope,
\]
where we used the fact that there is no cancellation when computing the sum of
squares and that the square-root preserves small relative errors for $\epsilon$
sufficiently small. 

Even though the relative errors in $z$ are complex, the relative error between
$n$ and $\un$ is real and positive, so the branch condition on $n$ is
preserved. Applying a similar analysis to the output variables in both branches
yields the following,
\[
\hu = \uhu + O(\epsilon),
\quad
\hv = \uhv\ope,
\quad
\hc = \uhc\ope,
\quad
\hs = \uhs\ope,
\]
\[
\hsigma = \uhsigma\ope,
\quad
\hgamma = \uhgamma\ope.
\]
Apart from $\hu$, the relative errors introduced by projecting the
nearly-unitary inputs resulted in similar relative errors between these outputs
and the exactly unitary ones. Even with the absolute error in $\hu$, this shows
that for sufficiently small $\epsilon$, the outputs of the symmetric turnover
applied to nearly-unitary inputs are also nearly-unitary. Moreover, the
corresponding core transformations satisfy 
\[
\left\|\uhQ\,\uhU\,\uhtQ - \hQ\hU\htQ \right\| = O(\epsilon).
\]
Combining this with the previous results gives
\[
\left\|\tQ U Q - \hQ \hU \htQ \right\| = O(\epsilon),
\]
which completes the proof.
\end{proof}

Theorem~\ref{thm:nearly} can be seen as a type of mixed stability result for
the symmetric turnover for data that is only nearly-unitary. Namely, there
exist nearby exactly unitary inputs whose outputs are close to the outputs
computed using the nearly-unitary data.

\subsection{Two formulas for $z$}
We conclude this section with an explanation of why the alternative formula for
$z$ proposed by Gragg is needed for stability, even when computing in exact
arithmetic. Recall that in the proof of Corollary~\ref{cor:symto} we used the
fact that $n\neq 0$ implies $z \neq 0$ to show that $\hsigma$ is always
well-defined. For inputs that are only nearly-unitary, the formula $z = 1 + t$
does not guarantee this in general. Take, for example,
\[
u = -1,              \quad
v = \sqrt{\epsilon}, \quad
c = s = 1/\sqrt{2},  \quad
\sigma = 1,          \quad
\gamma = 1,          \quad
\rho = -1,        
\]
for $\epsilon > 0$. Then, $t = \gammaB \sigmaB^2 u = -1$ and $ z = 1 + t = 0$,
but $n = \sqrt{c^2|z|^2 + v^2} = \sqrt{\epsilon}$. So $z = 0$ even though $n >
0$.  Since $\re(t) < 0$ we should instead use the alternative formula for $z$, 
\[
z = \frac{v^2 + 2\im(t)i}{1-\tB} = \frac{\epsilon}{2}.
\]
Using this definition, $n \neq 0 \implies z \neq 0$ as expected. 

For unitary core transformations, both formulas for $z$ are equal, so the
instability can only be seen for nearly-unitary inputs. As this example
illustrates, $z=0$ must imply $v=0$ so that $n$ is also 0. The alternative
formula for $z$ guarantees that this cannot happen, thus avoiding this potential
instability. 

\section{A Square-Root-Free Symmetric Turnover}
\label{sec4_rfuqr}

In this section we show how to compute the symmetric turnover from
Corollary~\ref{cor:symto} implicitly without computing any square roots. This
result is similar to the square-root-free QR algorithms for symmetric
tridiagonal matrices \cite{ortega1963, feng1991} and to the rational QR
algorithms for unitary matrices \cite{gragg1999, stewart2006}.

In this section we also show how to control the growth of the unitary defects
of the iteration parameters throughout the core-chasing steps. This involves
projecting the nearly-unitary outputs of the symmetric and square-root-free
turnovers onto nearby, nearly-unitary outputs with smaller defects. Since the
defects are small, this projection can be done without computing any
square roots using a degree-one Taylor polynomial.   

\subsection{Casting out the square roots}
In the $n \neq 0$ branch, the symmetric turnover in Corollary~\ref{cor:symto}
computes two square roots,
\[
n = \sqrt{c^2|z|^2 + v^2},
\qquad
|z| = \sqrt{\re(z)^2 + \im(z)^2}.
\]
To avoid computing these square roots, we should work with $n^2$ and $|z|^2$
directly, which requires that we square all variables that depend on $n$ and
$|z|$. For $n \neq 0$ this gives,
\begin{equation*}
\begin{array}{l}
\begin{array}{lll}
t &=& \gammaB\sigmaB^2 u,
\end{array} \\
\begin{array}{lll}
z &=& 
\begin{cases}
1+t, & \re(t) \geq 0,\\
\frac{v^2 + 2\im(t)i}{1-\tB}, & \re(t) < 0,
\end{cases}
\end{array} \\
\begin{array}{lll}
n^2 &=& c^2|z|^2 + v^2,\\
\hc^2 &=& c^2|z|^2/n^2,\\
\hsigma^2 &=& \rho^2 \sigma^2 z^2/|z|^2,\\
\hs^2 &=& v^2/n^2,\\
\hu &=& \rhoB(s^2u - \gamma\sigma^2c^2),\\
\hv^2 &=& s^2 n^2\\
\hgamma &=& -\rhoB\gamma.
\end{array}
\end{array}
\end{equation*}
Notice that $t$, $z$, $\hu$ and $\hgamma$ are all unchanged. For the remaining
variables, we compute their squares using $n^2$, $|z|^2$ and the squares of
the inputs. We can simplify this further if we combine $\gamma$ and
$\sigma$ into a new variable $\omega = \gamma\sigma^2$. Defining the
corresponding output $\homega = \hgamma \hsigma^2$, the symmetric turnover
reduces to
\begin{equation*}
\begin{array}{l}
\begin{array}{lll}
t &=& \omegaB u,
\end{array} \\
\begin{array}{lll}
z &=& 
\begin{cases}
1+t, & \re(t) \geq 0,\\
\frac{v^2 + 2\im(t)i}{1-\tB}, & \re(t) < 0,
\end{cases}
\end{array} \\
\begin{array}{lll}
n^2 &=& c^2|z|^2 + v^2,\\
\hc^2 &=& c^2|z|^2/n^2,\\
\homega &=& -\rho \omega z^2/|z|^2,\\
\hs^2 &=& v^2/n^2,\\
\hu &=& \rhoB(s^2u - c^2\omega),\\
\hv^2 &=& s^2 n^2.
\end{array}
\end{array}
\end{equation*}
The phase tracking variable $\gamma$ has been absorbed by $\omega$ and is no
longer explicitly needed. The following corollary shows that computing the
symmetric turnover without square roots is equivalent to computing the turnover
using square roots for a corresponding set of unsquared inputs.  

\begin{corollary}[Square-Root-Free Symmetric Turnover]\label{cor:rfsymto}
Given unimodular complex numbers $\rho$, $\gamma$  and $\omega$, a complex
number $u$, and nonnegative real numbers $\vv$, $\cc$ and $\sss$ such that
$|u|^2 + \vv = 1$ and $\cc + \sss = 1$, define the following, 
\begin{center}
\begin{tikzpicture}[scale=1.00,y=-5mm,x=5mm]
\node at (0.0,0.0) [align=left, anchor=west] {$t$};
\node at (1.0,0.0) [align=left, anchor=west] {$=$};
\node at (2.0,0.0) [align=left, anchor=west] {$\omegaB u,$};

\node at (0.0,2.0) [align=left, anchor=west] {$z$};
\node at (1.0,2.0) [align=left, anchor=west] {$=$};
\draw[
  decorate,
  decoration={brace, amplitude=6pt}
] (2.5,3.5) -- (2.5,0.5);
\node at (2.5,1.0) [align=left, anchor=west] {$1+t,$};
\node at (7.0,1.0) [align=left, anchor=west] {$\re(t) \geq 0,$};
\node at (2.5,2.5) [align=left, anchor=west] {$\dfrac{\vv + 2\im(t) i}{1-\tB},$};
\node at (7.0,2.5) [align=left, anchor=west] {$\re(t) < 0,$};

\node at (0.0,4.0) [align=left, anchor=west] {$\nn$};
\node at (1.0,4.0) [align=left, anchor=west] {$=$};
\node at (2.0,4.0) [align=left, anchor=west] {$\cc\cdot |z|^2 + \vv,$};

\node at (0.0,5.0) [align=left, anchor=west] {$\mathrm{if}\ \nn \neq 0:$};

\node at (1.0,6.0) [align=left, anchor=west] {$\hcc$};
\node at (2.0,6.0) [align=left, anchor=west] {$=$};
\node at (3.0,6.0) [align=left, anchor=west] {$\cc\cdot |z|^2/\nn,$};

\node at (1.0,7.0) [align=left, anchor=west] {$\hsss$};
\node at (2.0,7.0) [align=left, anchor=west] {$=$};
\node at (3.0,7.0) [align=left, anchor=west] {$\vv/\nn,$};

\node at (1.0,8.0) [align=left, anchor=west] {$\hu$};
\node at (2.0,8.0) [align=left, anchor=west] {$=$};
\node at (3.0,8.0) [align=left, anchor=west] {$\rhoB(\sss\cdot u - \cc\cdot\omega),$};

\node at (1.0,9.0) [align=left, anchor=west] {$\hvv$};
\node at (2.0,9.0) [align=left, anchor=west] {$=$};
\node at (3.0,9.0) [align=left, anchor=west] {$\sss\cdot \nn,$};

\node at (1.0,10.0) [align=left, anchor=west] {$\homega$};
\node at (2.0,10.0) [align=left, anchor=west] {$=$};
\node at (3.0,10.0) [align=left, anchor=west] {$-\rho \omega z^2/|z|^2,$};

\node at (0.0,11.0) [align=left, anchor=west] {$\mathrm{else}:$};

\node at (1.0,12.0) [align=left, anchor=west] {$\hcc$};
\node at (2.0,12.0) [align=left, anchor=west] {$=$};
\node at (3.0,12.0) [align=left, anchor=west] {$0,$};

\node at (1.0,13.0) [align=left, anchor=west] {$\hsss$};
\node at (2.0,13.0) [align=left, anchor=west] {$=$};
\node at (3.0,13.0) [align=left, anchor=west] {$1,$};

\node at (1.0,14.0) [align=left, anchor=west] {$\hu$};
\node at (2.0,14.0) [align=left, anchor=west] {$=$};
\node at (3.0,14.0) [align=left, anchor=west] {$\rhoB u,$};

\node at (1.0,15.0) [align=left, anchor=west] {$\hvv$};
\node at (2.0,15.0) [align=left, anchor=west] {$=$};
\node at (3.0,15.0) [align=left, anchor=west] {$0,$};

\node at (1.0,16.0) [align=left, anchor=west] {$\homega$};
\node at (2.0,16.0) [align=left, anchor=west] {$=$};
\node at (3.0,16.0) [align=left, anchor=west] {$1.$};
\end{tikzpicture}
\end{center}
Then, for
\[
u,
\quad
v = \sqrt{\vv},
\quad
c = \sqrt{\cc},
\quad
s = \sqrt{\sss},
\quad
\sigma = \sqrt{\gammaB\omega},
\quad
\gamma,
\quad
\rho,
\]
and
\[
\hu,
\quad
\hv = \sqrt{\hvv},
\quad
\hc = \sqrt{\hcc},
\quad
\hs = \sqrt{\hsss},
\quad
\hsigma = \sqrt{\hgammaB\homega},
\quad
\hgamma = -\rhoB\gamma,
\]
where the square roots are chosen so that the real variables are nonnegative
and $\hsigma = -\rho \sigma z/|z|$,
\[
\tQ U Q = \hQ \hU \htQ.
\]
\end{corollary}
\begin{proof}
The result follows directly from Corollary~\ref{cor:symto}.
\end{proof}

Corollary~\ref{cor:rfsymto} says that we can implicitly compute the symmetric
turnover without square roots if we work with $v^2$, $c^2$, $s^2$ and $\gamma
\sigma^2$ directly. While this saves two expensive square-root computations, it
comes at a cost. Namely, we no longer store the $U_i$ or $Q_i$ explicitly, and
without $Q_i$ we cannot compute eigenvectors. When calculating eigenvalues,
$v_i^2 \rightarrow 0$, and the eigenvalues are recovered using the $u_i$. Since
the $u_i$ are computed explicitly, the eigenvalues can be computed without
using square roots.

\subsection{Square-root-free renormalization}
In the original unitary QR algorithm in \cite{gragg1986} Gragg noted an
instability caused by a gradual increase in the unitary defect for the pair
$u_i$ and $v_i$. A practical solution is to project the
parameters back onto the unit circle after every iteration. If the unitary
defect is small, then the projection errors are small and relative,
\[
\euv = |u|^2 + v^2 - 1, 
\quad 
u \leftarrow \frac{u}{\sqrt{1+\euv}},
\quad 
v \leftarrow \frac{v}{\sqrt{1+\euv}}.
\] 
Theorem~\ref{thm:nearly} says that the outputs of the symmetric turnover are
nearly-unitary whenever the inputs are. Although not given explicitly, even for
inputs with unitary defects on the order of machine precision, the defects of
the output variables are, in the worst case, on the order of 10 times larger. 

Since we also need to renormalize in the square-root-free turnover, we would
prefer to avoid computing these additional square roots if possible. One option
is to compute the defects and check if they are below some threshold. If they
are small enough, the renormalization step can be skipped.  Otherwise,
renormalization is performed using square roots. While this may save computing
some of the square roots, there is no guarantee that they can be avoided in
general. A more universal approach is to use the fact that the defects are
small and approximate $1/\sqrt{1+\euv}$ with a degree-one Taylor polynomial.
For $\euv$ sufficiently small, the Taylor remainder theorem says that 
\[
\frac{1}{\sqrt{1+\euv}} = 1 - \euv/2 + O(\euv^2).
\]
So, if $u$ and $v$ are nearly-unitary with defect $\euv$, then
\[
u' = u(1-\euv/2),
\qquad
v' = v(1-\euv/2),
\]
are nearly-unitary with defect $O(\euv^2)$.

For defects on the order of machine precision, renormalizing using a degree one
Taylor polynomial is sufficient for reducing any growth in the defects to an
acceptable level. Even in exact arithmetic, we cannot expect the defects to be
smaller than the unit roundoff since $u'$ and $v'$ will also be approximated by
floating-point numbers.

This formula for renormalization requires computing the defect $\euv = |u|^2 +
v^2 - 1$, which is susceptible to cancellation. An alternative formula that
avoids this cancellation first computes $x = |u|^2 + v^2$ then renormalizes via
the formulas, 
\[
u' = u(3-x)/2,
\qquad
v' = v(3-x)/2.
\]
For the square-root-free turnover, we have to renormalize $vv$ instead of $v$.
Using a degree-one Taylor polynomial, this becomes
\[
1/(1+\euv)= 1-\euv + O(\euv^2),
\qquad
vv' = vv(1-\euv) = vv(2-x),
\]
where $x = |u|^2 + vv$ and $\euv = x-1$. Similarly, first order renormalization for 
the variables $\gamma$, $\sigma$ and $\omega$ as follows
\[
\gamma' = \gamma(3-|\gamma|^2)/2,
\quad
\sigma' = \sigma(3-|\sigma|^2)/2
\quad
\mathrm{and}
\quad
\omega' = \omega(3-|\omega|^2)/2.
\]
The variables $c$ and $s$, as well as $\cc$ and $\sss$, are normalized by
construction and therefore do not need to be renormalized. The results of
Theorem~\ref{thm:nearly} still apply for the turnovers with renormalization
since the first order corrections introduce at most $O(\epsilon)$ relative
errors. Algorithms~\ref{alg:symto} and \ref{alg:rfsymto} give pseudocode for
computing the symmetric and square-root-free turnovers with renormalization.
\begin{figure}[h!]
\centering
\begin{minipage}[t]{0.48\textwidth}
\captionsetup{type=algorithm}
\captionof{algorithm}{Symmetric Turnover with Renormalization}
\label{alg:symto}
\scriptsize
\begin{algorithmic}
\REQUIRE $u$, $v$, $c$, $s$, $\sigma$, $\gamma$, $\rho$
\STATE $t = \gammaB\sigmaB^2u$
\IF{$\re(t) \geq 0$}
  \STATE $z = 1+t$
\ELSE
  \STATE $z = (v^2 + 2\im(t)i)/(1-\tB)$
\ENDIF
\STATE $n = \sqrt{c^2 |z|^2 + v^2}$
\IF{$n > 0$}
  \STATE $\hc = c |z|/n$
  \STATE $\hsigma = -\rho \sigma z/|z|$
  \STATE $\hs = v/n$
  \STATE $\hu = \rhoB(s^2u-\gamma\sigma^2c^2)$
  \STATE $\hv = ns$
\ELSE
  \STATE $\hc = 0$
  \STATE $\hsigma = 1$
  \STATE $\hs = 1$
  \STATE $\hu = \rhoB u$
  \STATE $\hv = 0$
\ENDIF
\STATE $\hgamma = -\rhoB \gamma$
\STATE \textbf{Renormalization}
\STATE $x = |\hu|^2 + \hv^2$
\STATE $\hu = \hu(3-x)/2$
\STATE $\hv = \hv(3-x)/2$
\STATE $\hsigma = \hsigma(3-|\hsigma|^2)/2$
\STATE $\hgamma = \hgamma(3-|\hgamma|^2)/2$
\RETURN $\hu$, $\hv$, $\hc$, $\hs$, $\hsigma$, $\hgamma$
\end{algorithmic}
\end{minipage}
\hfill
\begin{minipage}[t]{0.48\textwidth}
\captionsetup{type=algorithm}
\captionof{algorithm}{Square-Root-Free Symmetric Turnover with Renormalization}
\label{alg:rfsymto}
\scriptsize
\begin{algorithmic}
\REQUIRE $u$, $\vv$, $\cc$, $\sss$, $\omega$, $\rho$
\STATE $t = \omegaB u$
\IF{$\re(t) \geq 0$}
  \STATE $z = 1+t$
\ELSE
  \STATE $z = (\vv + 2\im(t)i)/(1-\tB)$
\ENDIF
\STATE $zz = \re(z)^2 + \im(z)^2$
\STATE $nn = \cc\cdot zz + \vv$
\IF{$nn > 0$}
  \STATE $\hcc = \cc\cdot zz/nn$
  \STATE $\homega = -\rho\omega z^2/zz$
  \STATE $\hsss = \vv/nn$
  \STATE $\hu = \rhoB(\sss\cdot u-\cc\cdot\omega)$
  \STATE $\hvv = nn\cdot \sss$
\ELSE
  \STATE $\hcc = 0$
  \STATE $\homega = 1$
  \STATE $\hsss = 1$
  \STATE $\hu = \rhoB u$
  \STATE $\hvv = 0$
\ENDIF
\STATE \textbf{Renormalization}
\STATE $x = |\hu|^2 + \hvv$
\STATE $\hu = \hu(3-x)/2$
\STATE $\hvv = \hvv(2-x)$
\STATE $\homega = \homega(3-|\homega|^2)/2$
\RETURN $\hu$, $\hvv$, $\hcc$, $\hsss$, $\homega$
\end{algorithmic}
\end{minipage}

\end{figure}

\section{Numerical Stability}
\label{sec5_stability}

In this section we will show that the symmetric and square-root-free turnovers
in Algorithms~\ref{alg:symto} and \ref{alg:rfsymto} are backward stable, which
is enough to guarantee that the entire core-chasing iteration is backward
stable \cite{aurentz2018}. We will also show in Example~\ref{ex:nearly} that
even in exact arithmetic we cannot expect the turnover to be exact for matrices
that are only nearly-unitary. Since the floating-point representations of the
core transformations are almost never exactly unitary, we cannot expect any
sort of structured backward error result, at least not in a strict sense.
Instead we will prove normwise backward stability for the input
core transformations to the turnover. 

\subsection{Forward stability}
The first result is a mixed forward-backward error result when computing the
turnover in floating-point arithmetic. Our analysis assumes that the standard
axioms of floating-point arithmetic hold. See \cite{higham2002}, for example.
Let $\rerr$ be the unit roundoff. Then, for floating-point numbers $x$ and $y$,
possibly complex, we assume the following, 
\begin{equation*}
\begin{array}{rcll}
\fl(x \circ y) &=& (x \circ y)(1+O(\rerr)), & \circ = +,-,\ast,/ \\
\fl(\sqrt{x}) &=& \sqrt{x}\ (1+O(\rerr)).     &
\end{array}
\end{equation*}
The following lemma shows that both the symmetric and square-root-free
turnovers are forward stable.
\begin{lemma}[Forward Stability of Algorithm~\ref{alg:symto}]\label{lem:fstab}
Assume the matrices $\tQ$, $U$ and $Q$ are stored in floating point and
nearly-unitary with $O(\rerr)$ defects, and let $\hQ'$, $\hU'$ and $\htQ'$ be
the outputs of the symmetric turnover computed via Algorithm~\ref{alg:symto}.
Then,
\[
\left\|\tQ U Q - \hQ' \hU' \htQ' \right\| = O(\rerr).
\]
\end{lemma}
\begin{proof}
The first step is to compute $t$ from the inputs $u$, $\sigma$ and $\gamma$.
Since this is done via complex multiplication, it follows that
\[
t' = \fl(t) = \fl(\gammaB \sigmaB^2 u) = t\opr.
\]
Unlike Theorem~\ref{thm:nearly}, the relative error $\opr$ is complex in
general. So, even though the relative errors in $t'$ are small, the relative
errors in $\re(t')$ and $\im(t')$ may be large. This means that $t'$ does not
necessarily preserve the branch structure for computing $z$, except in the case
$t = 0$. 

To avoid this complication, the forward error in $t'$ can be interpreted as a
backward error for the input $\sigma$. Assume $t \neq 0$ and define $\sigma''$
as
\[
\sigma'' = \sigma \sqrt{\tB'/\tB}.
\]
Then, $\sigma'' = \sigma\opr$ and $t' = \gammaB \sigmaB''^2 u$. Next, define
$\tQ''$ and $Q''$ using $\sigma''$ and the other unperturbed inputs,
\[
\tQ'' = 
\begin{bmatrix}
-\rhoB \gamma c \sigma'' & \phantom{-{}}\rhoB s  & \\
-\rho s & -\rho\gammaB c\sigmaB'' & \\
     &                & 1 
\end{bmatrix},
\qquad
Q'' = 
\begin{bmatrix}
 c\sigma'' & -s           & \\
 s & \phantom{-{}}c\sigmaB'' \\
     &                & 1 
\end{bmatrix}.
\]
It follows immediately that 
\[ 
\| \tQ U Q - \tQ'' U Q'' \| = O(\rerr).
\]
Applying Algorithm~\ref{alg:symto} to the inputs $u$, $v$, $c$, $s$,
$\sigma''$, $\gamma$ and $\rho$ in exact arithmetic gives outputs $\hu''$,
$\hv''$, $\hc''$, $\hs''$, $\hsigma''$ and $\hgamma''$, and corresponding
matrices $\hQ''$, $\hU''$ and $\htQ''$ that satisfy
\[
\| \tQ'' U Q'' - \hQ'' \hU'' \htQ''\| = O(\rerr).
\]
To complete the proof it is enough to show that 
\[
\hu' = \hu'' + O(\rerr),
\quad
\hv' = \hv''\opr,
\quad
\hc' = \hc''\opr,
\]
\[
\hs' = \hs''\opr,
\quad
\hsigma' = \hsigma''\opr,
\quad
\hgamma' = \hgamma''\opr,
\]
from which it follows that
\[
\| \hQ'' \hU'' \htQ'' - \hQ' \hU' \htQ'\| = O(\rerr).
\]
In exact arithmetic, the first step in the turnover using $\sigma''$ is
\[
t'' = \gammaB \sigmaB''^2 u = t'.
\]
So using $\sigma''$ in exact arithmetic produces the same value for $t$ as
using $\sigma$ in floating-point arithmetic. This means that the branching
condition for $z$ is maintained. For $\re(t') \geq 0$,
\[
z' = \fl(z) = \fl(1+t') = (1+t')\opr = (1+t'')\opr = z''\opr.
\] 
For the branch $\re(t') < 0$, 
\[
z' = \frac{\fl(v^2) + \fl(2\im(t'))i}{\fl(1-\tB')} 
= \frac{v^2\opr + 2\im(t')\opr i}{(1-\tB')\opr}. 
\]
Since $v^2$ and $2\im(t')$ are real, $v^2$ and $2\im(t')i$ are orthogonal, so
their rounding errors are real and orthogonal and there is no cancellation.
Therefore, the rounding errors in $z'$ satisfy 
\[
z' = \frac{v^2 + 2\im(t')i}{1-\tB'}\opr = z''\opr. 
\]
Applying a similar analysis to $n'$, $\hv'$, $\hc'$, $\hs'$ and $\hgamma'$
gives
\[
n' = n''\opr,
\quad
\hv' = \hv''\opr,
\quad
\hc' = \hc''\opr,
\]
\[
\hs' = \hs''\opr,
\quad
\hgamma' = \hgamma''\opr,
\]
before renormalization. For $\hu'$ and $\hsigma'$, extra care has to be taken
to account for the difference between $\sigma$ and $\sigma''$,
\begin{equation*}
\begin{split}
\hu' &= \fl(\rhoB(s^2u - \gamma \sigma^2 c^2))\\ 
&= \rhoB(s^2u - \gamma \sigma^2 c^2) + O(\rerr)\\
&= \rhoB(s^2u - \gamma \sigma''^2\opr c^2) + O(\rerr)\\
&= \hu'' + O(\rerr),
\end{split}
\end{equation*}
\[
\hsigma' = \fl(-\rho \sigma z'/|z'|) 
= -\rho \sigma'' z''/|z''|\opr = \hsigma''\opr.
\]
Thus, before renormalization, the $'$ and $''$ variables are $O(\rerr)$ close.
Since they are also nearly-unitary with $O(\rerr)$ defects, the errors
introduced from renormalization, including the additional rounding errors, are
relative and $O(\rerr)$, and the outputs remain within $O(\rerr)$ of each other. 
\end{proof}

Lemma~\ref{lem:fstab} says that computing Algorithm~\ref{alg:symto} in
floating-point arithmetic gives a factorization $\hQ' \hU' \htQ'$ that is
still within $O(\rerr)$ of the product of the inputs $\tQ U Q$ as long as their
nearly-unitary defects are also $O(\rerr)$. As the following corollary shows,
the same is true for the square-root-free turnover in
Algorithm~\ref{alg:rfsymto}.  

\begin{corollary}[Forward Stability of Algorithm~\ref{alg:rfsymto}]\label{cor:fstab}
Assume the matrices $\tQ$, $U$ and $Q$ are stored in floating point and
nearly-unitary with $O(\rerr)$ defects, and let $\hQ'$, $\hU'$ and $\htQ'$ be
the outputs of the symmetric turnover computed implicitly via
Algorithm~\ref{alg:rfsymto}. Then,
\[
\left\|\tQ U Q - \hQ' \hU' \htQ' \right\| = O(\rerr).
\]
\end{corollary}
\begin{proof}
The proof follows from Lemma~\ref{lem:fstab} and the arguments in
Theorem~\ref{thm:nearly} and Corollary~\ref{cor:rfsymto}.
\end{proof}

\subsection{Backward stability}

We are now ready to show backward stability of the symmetric and
square-root-free turnovers. The following theorem shows that the previous
forward error results can be interpreted as backward errors in the input
matrices, as long as they are nearly-unitary. 

\begin{theorem}[Backward Stability of Algorithms~\ref{alg:symto} and
\ref{alg:rfsymto}]\label{thm:bstab}
Let $\tQ$, $U$ and $Q$, and $\hQ'$, $\hU'$ and $\htQ'$ be as in
Lemma~\ref{lem:fstab} and Corollary~\ref{cor:fstab}. Then, there exists $U'$
such that, 
\[
\tQ U' Q = \hQ' \hU' \htQ',
\qquad
\|U - U'\| = O(\rerr).
\]
\end{theorem}
\begin{proof}
By Lemma~\ref{lem:fstab}, there exists $\Delta$ such that
\[
\tQ U Q + \Delta = \hQ' \hU' \htQ',
\qquad
\|\Delta\| = O(\rerr).
\]
Since $\tQ$ and $Q$ are nearly-unitary, $\tQ^{-1}$ and $Q^{-1}$ exist and are also 
nearly-unitary with $O(\rerr)$ defects. Define $U'$ as,
\[
U' = U + \tQ^{-1} \Delta Q^{-1}.
\]
Then, 
\[
\tQ U' Q = \hQ' \hU' \htQ',
\qquad
\|U - U'\| = O(\rerr),
\]
which completes the proof.
\end{proof}
Theorem~\ref{thm:bstab} says that computing the symmetric turnover in
floating-point arithmetic for inputs that are nearly-unitary with $O(\rerr)$
defects computes a factorization that is exact for a nearby input matrix $U'$.
This means that the symmetric turnover and the entire core-chasing step are
backward stable \cite{aurentz2018}. While $U$ is a core transformation, $U'$ is
almost surely not, which means the errors in $U-U'$ are not confined to the
nontrivial block $[U]$. Example~\ref{ex:nearly} shows that we cannot expect
structured errors for nearly-unitary matrices, so this result is not
surprising.

\subsection{Structured backward error}
The question remains as to whether the backward stability in
Theorem~\ref{thm:bstab} can be further refined into a form of structured
backward stability. In the strictest sense this would mean mapping the forward
errors in the computed outputs to backward errors on the input parameters that
define the core transformations. The following example shows that the symmetric
turnover is not, in general, exact for nearly-unitary matrices.  This
immediately implies that for computed outputs that are only nearly-unitary,
there is no way to convert forward errors in the outputs of the turnover into
backward errors on the inputs.  


\begin{example}[Impossibility of Structured Backward Error for Nearly-Unitary
Outputs]\label{ex:nearly}
For $0 < \epsilon < 1$, define the following nearly-unitary output parameters
of the symmetric turnover,
\[
\hu = 1,                 \quad
\hv = \sqrt{2\epsilon},  \quad
\hc = \hs = 1/\sqrt{2},  \quad
\hsigma = 1,             \quad
\hgamma = 1,             \quad
\rho = -1.
\]
These parameters are all nearly-unitary, so the following product is also
nearly-unitary,
\[
\hQ \hU \htQ
=
\begin{bmatrix}
1 & -\sqrt{\epsilon} & \sqrt{\epsilon}\\
\sqrt{\epsilon} & 0 & -1\\
\sqrt{\epsilon} & 1 & 0
\end{bmatrix}.
\]
Suppose there exist input parameters $u$, $v$, $c$, $s$, $\sigma$ and $\gamma$
such that
\[
\tQ U Q = \hQ \hU \htQ.
\]
Expanding the product of the input factors gives
\[
\tQ U Q
=
\begin{bmatrix}
\rhoB(s^2u-\gamma c^2\sigma^2)
&
\rhoB cs(\gamma\sigma+\sigmaB u)
&
-\rhoB sv
\\[1ex]
-\rho cs(\sigma+\gammaB\sigmaB u)
&
\rho s^2-\rho\gammaB c^2\sigmaB^2u
&
\rho\gammaB c\sigmaB v
\\[1ex]
sv
&
c\sigmaB v
&
\uB
\end{bmatrix}.
\]
Comparing the $(3,3)$ entries gives $\uB=0$, and hence $u=0$. Comparing
the $(2,2)$ entries then gives $\rho s^2=0$. Since $\rho=-1$, this implies
$s=0$. However, comparing the $(1,3)$ entries gives
\[
-\rhoB sv=\sqrt{\epsilon}.
\]
Since $0<\epsilon<1$, the right-hand side is nonzero, contradicting $s=0$.
Therefore there exist no input parameters such that
\[
\tQ U Q = \hQ \hU \htQ.
\]
In particular, the nearly-unitary output factors cannot, in general, be
interpreted as the exact symmetric turnover of nearby nearly-unitary input
core transformations.
\end{example}

Example~\ref{ex:nearly} shows that we cannot expect the symmetric turnover to
be exact when working with inputs that are only nearly-unitary, even when
computed in exact arithmetic. Since the forward errors caused by computing in
floating-point arithmetic will almost surely result in outputs that are only
nearly-unitary, we cannot expect to be able to find perturbed inputs that make
the turnover exact.

\section{Numerical Experiments}
\label{sec6_experiments}

In this section we present several numerical experiments to illustrate the
effects of removing the square roots from the computations. 

\subsection{Square roots are expensive}
The goal of our first experiment is to show why removing square roots is so
important. While it is true that the double-precision operations $+$, $\ast$,
$/$ and $\sqrt{\cdot}$ can all be computed with the same high relative
accuracy, the \emph{latency}, or number of clock cycles it takes to execute
each operation, can be quite different, even for modern processors. The results
published in \cite{fog2025} suggest that the operations $/$ and $\sqrt{\cdot}$
require on average 2--5 times more cycles than the operations $+$ and $\ast$ for
real, double-precision variables. 

Figure~\ref{fig:optimes} shows the total clock time for computing $10^8$
sequential operations of each type on an Intel Core i7-4790 processor. Here we
clearly see that $/$ and $\sqrt{\cdot}$ require more time on average than $+$
and $\ast$ when executed sequentially.
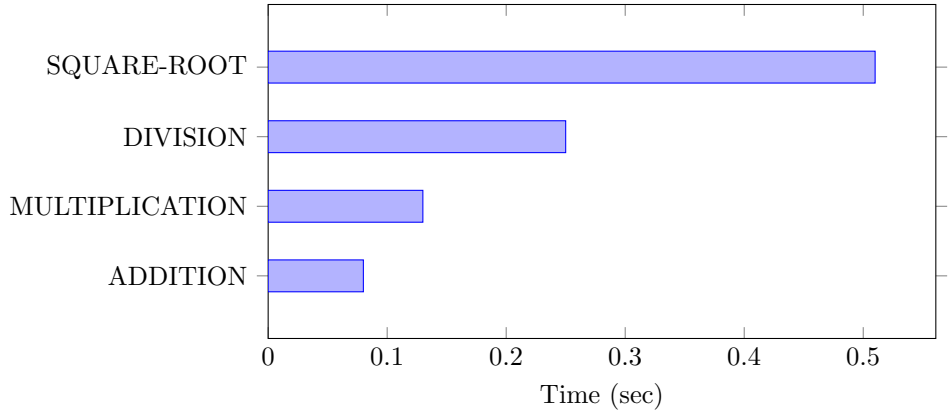
\begin{figure}[h!tbp]
\centering
\begin{tikzpicture}
\begin{axis}[
  xbar,
  xmin=0,
  ytick=data,
  yticklabels={
    ADDITION,
    MULTIPLICATION,
    DIVISION,
    SQUARE-ROOT
  },
  xlabel={Time (sec)},
  width=0.8\linewidth,
  height=6cm,
  bar width=12pt,
  enlarge y limits=0.3,
]
  \addplot coordinates {
    (0.08,0)
    (0.13,1)
    (0.25,2)
    (0.51,3)
  };
\end{axis}
\end{tikzpicture}
\caption{Total compute time in seconds for executing $10^8$ sequential, real,
double-precision operations of each type on an Intel Core i7-4790 processor.}
\label{fig:optimes}
\end{figure}

\subsection{Square-root-free unitary core-chasing}

Our second experiment compares the runtime and accuracy of using four different
methods to compute the roots of unity, i.e., the zeros of the polynomial $z^n =
1$. This can be reformulated as a unitary upper Hessenberg eigenvalue problem
using the companion matrix. The four methods we compare are the dense QR
algorithm for complex upper Hessenberg matrices \texttt{zhseqr} from LAPACK,
the unitary core-chasing algorithm from \cite{aurentz2015a}, denoted UQR, and
Algorithms~\ref{alg:symto} and \ref{alg:rfsymto}, denoted SYMUQR and RFUQR,
respectively. The code for the methods UQR, SYMUQR, and RFUQR is available at
\url{github.com/eiscor/eiscor}.

Figure~\ref{fig:roots} plots the average runtime and the worst forward error in
the computed roots of unity. From the plot of average runtime we can see the
difference between the $O(n^3)$ complexity of dense QR from LAPACK and the
$O(n^2)$ algorithms that exploit the special unitary upper Hessenberg
structure. From the plot of worst-case forward errors in the computed
eigenvalues we can see that all four methods have comparable accuracy,
computing approximate eigenvalues that are within a modest multiple of machine
precision of the true roots.  
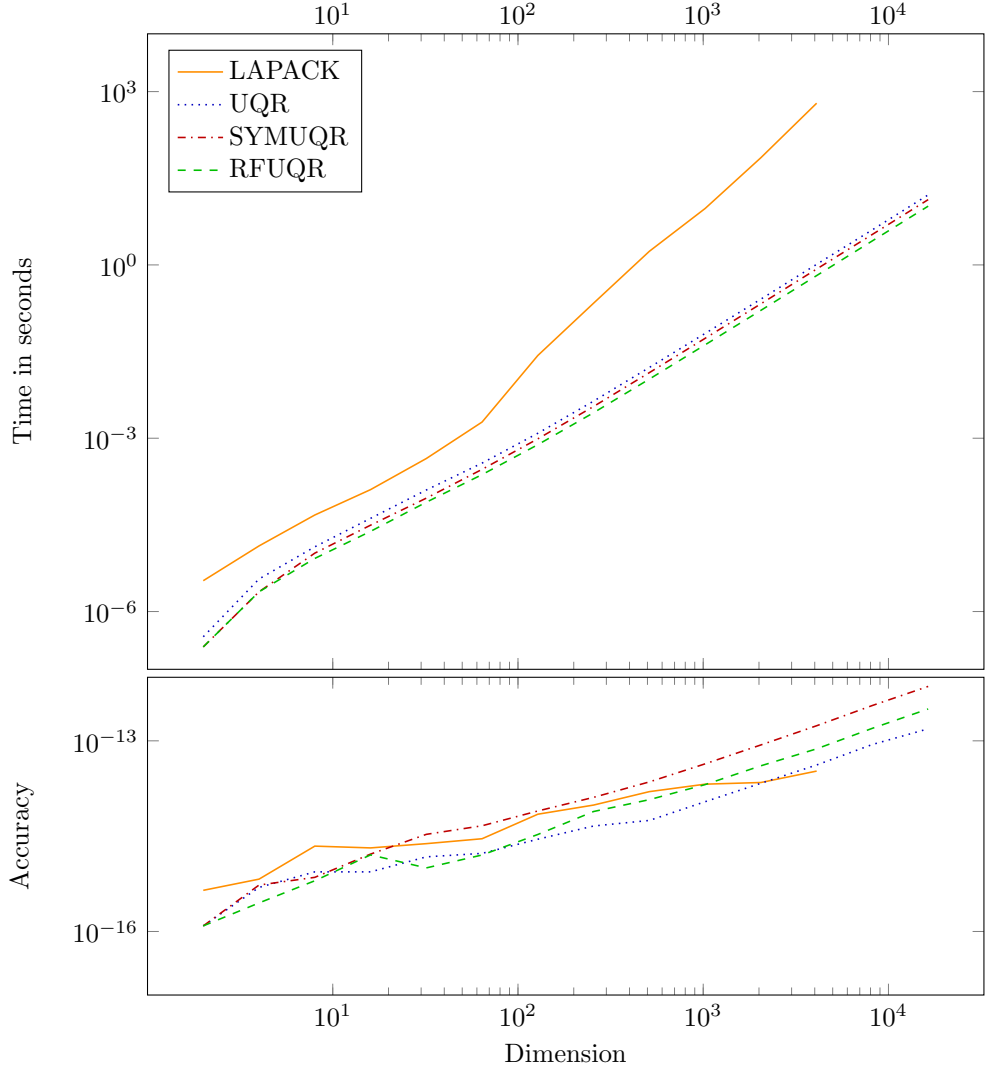
\begin{figure}[h!tbp]
  \centering
  \begin{tikzpicture}
    \begin{loglogaxis}[
      name=numplot1,
      scale only axis,
      xmin=1, xmax=32768,
      ymin=1e-7, ymax=1e4,
      xtick={1e1,1e2,1e3,1e4},
      ytick={1e-6,1e-3,1e0,1e3},
      xticklabel pos=right,
      ylabel={Time in seconds},
      every axis y label/.style= {at={(-0.15,0.5)},rotate=90},
      width=0.85\textwidth,
      height=0.40\textheight,
      axis on top,
      legend style = {
        fill opacity = 1,      
        draw opacity = 1,
        text opacity = 1,
        cells = {anchor = west},
        legend pos = north west,
        at = {(0.025,0.975)},
      },
      cycle list = {%
        SPECorange, semithick\\
        SPECblue, semithick, dotted\\
        SPECred, semithick, dashdotted\\
        SPECgreen, semithick, dashed\\
      },
      ]
      
 \addplot coordinates{ 
(      2 , 3.4180E-06 )%
(      4 , 1.3672E-05 )%
(      8 , 4.6875E-05 )%
(     16 , 1.2891E-04 )%
(     32 , 4.4531E-04 )%
(     64 , 1.9062E-03 )%
(    128 , 2.7062E-02 )%
(    256 , 2.1769E-01 )%
(    512 , 1.7212E+00 )%
(   1024 , 9.4750E+00 )%
(   2048 , 7.1748E+01 )%
(   4096 , 6.2936E+02 )%
 };

 \addplot coordinates{ 
(      2 , 3.6621E-07 )%
(      4 , 3.6621E-06 )%
(      8 , 1.3184E-05 )%
(     16 , 4.1016E-05 )%
(     32 , 1.2695E-04 )%
(     64 , 3.7109E-04 )%
(    128 , 1.2187E-03 )%
(    256 , 4.3906E-03 )%
(    512 , 1.6781E-02 )%
(   1024 , 6.5187E-02 )%
(   2048 , 2.5575E-01 )%
(   4096 , 1.0240E+00 )%
(   8192 , 4.0530E+00 )%
(  16384 , 1.6283E+01 )%
 };

 \addplot coordinates{ 
(      2 , 2.4414E-07 )%
(      4 , 2.1973E-06 )%
(      8 , 1.0254E-05 )%
(     16 , 3.1250E-05 )%
(     32 , 9.1797E-05 )%
(     64 , 2.8906E-04 )%
(    128 , 9.7656E-04 )%
(    256 , 3.5469E-03 )%
(    512 , 1.3656E-02 )%
(   1024 , 5.3625E-02 )%
(   2048 , 2.1225E-01 )%
(   4096 , 8.4825E-01 )%
(   8192 , 3.3755E+00 )%
(  16384 , 1.3485E+01 )%
 };

 \addplot coordinates{ 
(      2 , 2.4414E-07 )%
(      4 , 2.1973E-06 )%
(      8 , 8.3008E-06 )%
(     16 , 2.4414E-05 )%
(     32 , 7.8125E-05 )%
(     64 , 2.3437E-04 )%
(    128 , 7.7344E-04 )%
(    256 , 2.7656E-03 )%
(    512 , 1.0563E-02 )%
(   1024 , 4.1375E-02 )%
(   2048 , 1.6350E-01 )%
(   4096 , 6.5350E-01 )%
(   8192 , 2.6105E+00 )%
(  16384 , 1.0395E+01 )%
 };
 
      \legend{LAPACK, UQR, SYMUQR, RFUQR};
    \end{loglogaxis}
    
    \begin{loglogaxis}[
      at = {($(numplot1.south)+(0,-3pt)$)},
      anchor = north,
      scale only axis,
      xmin=1, xmax=32768,
      ymin=1e-17, ymax=1e-12,
      xtick={1e1,1e2,1e3,1e4},
      ytick={1e-16,1e-13},
      xlabel={Dimension},
      ylabel={Accuracy},
      every axis y label/.style= {at={(-0.15,0.5)},rotate=90},
      width=0.85\textwidth,
      height=0.20\textheight,
      axis on top,
      cycle list = {%
        SPECorange, semithick\\
        SPECblue, semithick, dotted\\
        SPECred, semithick, dashdotted\\
        SPECgreen, semithick, dashed\\
      },
      ]

 \addplot coordinates{ 
(      2 , 4.4409E-16 )%
(      4 , 6.6616E-16 )%
(      8 , 2.2037E-15 )%
(     16 , 2.0659E-15 )%
(     32 , 2.4101E-15 )%
(     64 , 2.8802E-15 )%
(    128 , 7.0027E-15 )%
(    256 , 9.7548E-15 )%
(    512 , 1.5870E-14 )%
(   1024 , 2.0795E-14 )%
(   2048 , 2.2094E-14 )%
(   4096 , 3.3426E-14 )%
 };

 \addplot coordinates{ 
(      2 , 1.2246E-16 )%
(      4 , 4.9388E-16 )%
(      8 , 8.6711E-16 )%
(     16 , 8.6711E-16 )%
(     32 , 1.4937E-15 )%
(     64 , 1.6947E-15 )%
(    128 , 2.8305E-15 )%
(    256 , 4.5685E-15 )%
(    512 , 5.5876E-15 )%
(   1024 , 1.1019E-14 )%
(   2048 , 2.1484E-14 )%
(   4096 , 4.1692E-14 )%
(   8192 , 8.7631E-14 )%
(  16384 , 1.5447E-13 )%
 };

 \addplot coordinates{ 
(      2 , 1.2246E-16 )%
(      4 , 5.3788E-16 )%
(      8 , 7.1089E-16 )%
(     16 , 1.6505E-15 )%
(     32 , 3.3766E-15 )%
(     64 , 4.6061E-15 )%
(    128 , 7.8779E-15 )%
(    256 , 1.2905E-14 )%
(    512 , 2.2580E-14 )%
(   1024 , 4.3451E-14 )%
(   2048 , 8.6164E-14 )%
(   4096 , 1.7280E-13 )%
(   8192 , 3.5994E-13 )%
(  16384 , 7.1828E-13 )%
 };
      
 \addplot coordinates{ 
(      2 , 1.2246E-16 )%
(      4 , 2.8140E-16 )%
(      8 , 6.2804E-16 )%
(     16 , 1.6041E-15 )%
(     32 , 1.0007E-15 )%
(     64 , 1.5983E-15 )%
(    128 , 3.3611E-15 )%
(    256 , 7.7238E-15 )%
(    512 , 1.1868E-14 )%
(   1024 , 2.0482E-14 )%
(   2048 , 4.0470E-14 )%
(   4096 , 7.4847E-14 )%
(   8192 , 1.5752E-13 )%
(  16384 , 3.1717E-13 )%
 };
 
    \end{loglogaxis}
  \end{tikzpicture}
  \caption{Runtime and accuracy for computing the roots of unity using four different algorithms.}
  \label{fig:roots}
\end{figure}

Table~\ref{tab:4096} gives the runtime and accuracy of the four methods for the
specific value $n = 4096$. From the last three columns it is clear that
exploiting the unitary upper Hessenberg structure greatly reduces the overall
runtime when compared to LAPACK. The last two columns also show that exploiting
symmetry and removing square roots can further reduce the total runtime, even
when compared to other methods that exploit this structure. 
\begin{table}[htbp]
\centering
\begin{tabular}{c|cccc}
$n=4096$  & LAPACK & UQR & SYMUQR & RFUQR \\
\hline
Runtime (sec) & $629$ & $1.02$ & $0.85$ & $0.65$ \\
Accuracy & $3.3 \times 10^{-14}$ & $4.2 \times 10^{-14}$ & $1.7 \times 10^{-13}$ & $7.5 \times 10^{-14}$
\end{tabular}
\caption{Runtime and accuracy of the four methods for the case $n=4096$.}
\label{tab:4096}
\end{table}

\section{Conclusions}
\label{sec7_conclusions}

In this article we presented a matrix-based derivation of a square-root-free
unitary core-chasing algorithm for computing eigenvalues of unitary
upper Hessenberg matrices. The derivation is based on a symmetry in the
turnover operation used in the unitary core-chasing algorithm from
\cite{aurentz2015a}, which reduces the arithmetic needed to perform the
core-chasing step and exposes the connection with square-root-free QR
algorithms for symmetric tridiagonal matrices. We showed that the na\"ive
symmetric turnover formulas can be unstable when the core transformations are
only nearly-unitary, as occurs in floating-point arithmetic, and introduced a
change of variables, analogous to Gragg's stabilized recurrence, that avoids
this cancellation and leads to a stable symmetric turnover for nearly-unitary
data. The same change of variables also makes it possible to derive a
square-root-free version of the algorithm. The floating-point analysis shows
that both the symmetric and square-root-free turnovers are stable for
nearly-unitary input, and that the resulting forward error can be converted
into a local backward error result for the turnover, although a strict
structured backward error result is impossible in general. Finally, the
numerical experiments demonstrate that the combination of exploiting the
symmetry in the turnover and casting out square roots gives accuracy
comparable to LAPACK and to the unitary core-chasing algorithm in
\cite{aurentz2015a}, while reducing the overall runtime.

\section*{Acknowledgments}
OpenAI ChatGPT was used during the preparation of this manuscript to assist
with checking mathematical derivations, exploring examples, literature
searches, and improving the presentation of the manuscript. All mathematical
arguments, references, computations, and conclusions were independently
reviewed and verified by the authors. The authors assume responsibility for all
content.

\bibliographystyle{siamplain}
\bibliography{rfuqr_references}

\end{document}